\documentclass[11pt,reqno]{amsart}
\usepackage[latin1]{inputenc} 
\usepackage{indentfirst}
\usepackage{amsmath} 
\usepackage{amsfonts} 
\usepackage{amsthm} 
\usepackage{epsfig} 
\usepackage{graphicx} 
\usepackage{multirow} 
\usepackage{verbatim} 
\usepackage{rotating} 
\usepackage{lettrine}
\usepackage{lscape}
\usepackage{ae}
\usepackage{float}
\usepackage[all]{xy}
\usepackage{icomma} 
\usepackage{amsmath} 
\usepackage{amsfonts} 
\usepackage{amsthm} 
\usepackage{varioref} 
\usepackage{layout}
\usepackage[Lenny]{fncychap} 
\usepackage{enumerate}
\usepackage{amsfonts, amsmath, amssymb, stmaryrd, latexsym, dsfont}
\usepackage{array}
\usepackage{longtable}
\usepackage{geometry}
\usepackage[frenchb,english]{babel}
\usepackage{mathrsfs}
\usepackage[mathscr]{euscript}

\begin{document}
\theoremstyle{plain}
\newtheorem{them}{Theorem}[]
\newtheorem{lem}{Lemma}[]
\newtheorem{propo}{Proposition}[]
\newtheorem{coro}{Corollary}[]
\newtheorem{proprs}{properties}[]
\theoremstyle{remark}
\newtheorem{rema}{Remark}[]
\newtheorem{remas}{Remarks}[]
\newtheorem{exam}{\textbf{Numerical Example}}[]
\newtheorem{exams}{\textbf{Numerical Examples}}[]
\newtheorem{df}{definition}[]
\newtheorem{dfs}{definitions}[]
\def\q{\mathbb{Q}(\sqrt{d})/\mathbb{Q}}
\def\K{\mathbb{Q}(\sqrt{d},i)}
\def\NN{\mathds{N}}
\def\RR{\mathbb{R}}
\def\HH{I\!\! H}
\def\QQ{\mathbb{Q}}
\def\CC{\mathbb{C}}
\def\ZZ{\mathbb{Z}}
\def\OO{\mathcal{O}}
\def\kk{\mathds{k}}
\def\KK{\mathbb{K}}
\def\ho{\mathcal{H}_0^{\frac{h(d)}{2}}}
\def\LL{\mathbb{L}}
\def\h{\overline{\mathcal{H}_0}}
\def\hh{\overline{\mathcal{H}_1}}
\def\hhh{\overline{\mathcal{H}_2}}
\def\jj{\overline{\mathcal{H}_3}}
\def\jjj{\overline{\mathcal{H}_4}}
\def\k{\mathds{k}^{(*)}}
\def\l{\mathds{L}}
\def\L{\kk_2^{(2)}}
\title[ Structure of $G$...]{Structure of $\mathrm{Gal}(\L/\kk)$ for some fields\\ $\kk=\QQ(\sqrt{2p_1p_2},i)$ with $\mathbf{C}l_2(\kk)\simeq(2, 2, 2)$}
\author[Abdelmalek AZIZI]{Abdelmalek Azizi}
\address{Abdelmalek Azizi and Abdelkader Zekhnini: Département de Mathématiques, Faculté des Sciences, Université Mohammed 1, Oujda, Morocco }
\author{Abdelkader Zekhnini}
\email{abdelmalekazizi@yahoo.fr}
\email{zekha1@yahoo.fr}

\author[Mohammed Taous]{Mohammed Taous}
\address{Mohammed Taous: Département de Mathématiques, Faculté des Sciences et Techniques, Université Moulay Ismail, Errachidia, Morocco}
\email{taousm@hotmail.com}
\maketitle
\selectlanguage{english}
\begin{abstract}
 Let $p_1 \equiv p_2 \equiv5\pmod8$ be different primes. Put $i=\sqrt{-1}$ and $d=2p_1p_2$, then the bicyclic biquadratic field $\kk=\K$ has an elementary abelian 2-class group of rank $3$.  In this paper  we determine the nilpotency class, the coclass, the generators and the structure of the non-abelian Galois group $\mathrm{Gal}(\L/\kk)$ of the second Hilbert 2-class field $\L$ of $\kk$. We study the  capitulation problem of the 2-classes of $\kk$ in its seven unramified quadratic extensions $\KK_i$ and in its seven unramified bicyclic biquadratic extensions $\LL_i$.
\end{abstract}
 \section{Introduction}
 Let $k$ be an algebraic number field and $Cl_2(k)$ be its $2$-class group i.e. the Sylow $2$-subgroup of the  ideal
class group, $Cl(k)$, of $k$.  Denote by $k_2^{(1)}$ the Hilbert 2-class field of $k$ and  by $k_2^{(2)}$ its second  Hilbert 2-class field. Put $G=\mathrm{Gal}(k_2^{(2)}/k)$ and let $G'$ denote its derived group, then it is well known that $C/G'\simeq Cl_2(k)$. The knowledge of $G$, its structure and its generators solves a lot of problems in number theory as capitulation problems, the finiteness or not of the towers  of number fields and the structures of the $2$-class groups of the unramified extensions of $k$ within $k_2^{(1)}$.  For particular types of fields k, for example,  fields with $Cl_2(k)\simeq(2, 2)$, the structure of $G$ has
been completely determined (see \cite{Ki76}). The success in this case is in part due to the fact
that, contrary to most other cases, 2-groups whose abelianization  is $(2, 2)$ are
well understood, cf. \cite{Ta37} and \cite{Se70}.\par
  If one considers another case, namely where $k=\K$ and  $Cl_2(k)\simeq(2, 2, 2)$,  for some square-free   integer $d$, then
the situation is very different and very difficult  when compared with the case described above; moreover  there is no known way (to our knowledge) to determine  the structure of $G$. Our aim in the present paper is to determine the isomorphism types of the second 2-class group of certain number fields $k=\K$, to give the structure and generators of $G$  and   we will explicitly determine   $\ker j_{K/k}$, the  kernel  of the natural class extension homomorphism $j_{K/k}: Cl(k) \longrightarrow Cl(K)$, where  $K$ is an unramified  extension of $k$ within $k_2^{(1)}$. It should be noted that the determination of $\ker j_{K/k}$ is not always easy to do, especially when $K=k(\sqrt{a+b\sqrt{-1}})$ or  $K=k(\sqrt{a+b\sqrt{-1}}, \sqrt{a'+b'\sqrt{-1}})$, with some positive integers $a$, $b$, $a'$ and  $b'$.\par
 Let $m$ be a square-free integer and $K$ be a number field. Throughout this paper, we adopt the following notations:
 \begin{itemize}
   \item $h(m)$, (resp. $h(K)$): the $2$-class number  of $\QQ(\sqrt m)$ (resp. $K$).
   \item $\mathcal{O}_K$: the ring of integers of $K$.
   \item $E_K$: the unit group of $\mathcal{O}_K$.
   \item  $W_K$: the group of roots of unity contained in $K$.
   \item $\omega_K$: the order of $W_K$.
   \item $K^+$: the maximal real subfield of $K$, if it is a CM-field.
   \item $Q_K=[E_K:W_KE_{K^+}]$ is Hasse's unit index, if $K$ is a CM-field.
  \item $q(K/\QQ)=[E_K:\prod_i^s E_{k_i}]$ is the unit index of $K$, if $K$ is multiquadratic, where $k_i$  are the quadratic subfields of $K$.
  \item $K^{(*)}$: the genus field of $K$.
  \item $\mathbf{C}l_2(K)$: the 2-class group of $K$.
  \item $\varepsilon_m$: the fundamental unit of $\QQ(\sqrt m)$.
  \item $i=\sqrt{-1}$.
  \item $\mathrm{FSU}$: denotes a fundamental system of units.
 \end{itemize}
\section{Main results}
Let $p_1\equiv p_2\equiv 5 \pmod8$ be different primes, then there exist some positive integers $e$, $f$, $g$ and $h$ such that  $p_1=e^2+4f^2$ and $p_2=g^2+4h^2$. Let  $p_1=\pi_1\pi_2$ and $p_2=\pi_3\pi_4$,  where $\pi_1=e+4if$ and  $\pi_2=e-4if$ (resp. $\pi_3=g+4ih$ and $\pi_4=g-4ih$) are conjugate prime elements in the cyclotomic field $k=\QQ(i)$ dividing $p_1$ (resp. $p_2$). Denote by $\kk$ the imaginary bicyclic biquadratic field $\K$, where $d=2p_1p_2$, its three quadratic subfields are $k=\QQ(i)$, $k_0=\QQ(\sqrt{d})$ and $\overline{k}_0=\QQ(\sqrt{-d})$. Let $\kk_2^{(1)}$ be the Hilbert 2-class field of $\kk$,  $\L$ its second Hilbert 2-class field and $G$ be the Galois group of $\kk_2^{(2)}/\kk$. According to  \cite{AT08},  $\kk$ has an elementary abelian 2-class group $\mathbf{C}l_2(\kk)$ of rank 3, that is, of type $(2, 2, 2$). In an earlier paper \cite{AZT-3} we have proved that the 2-class field tower of $\kk$ has length 2,  the order of $G$ is greater than or equal to  $64$,  we have given necessary and sufficient conditions to have $G$ of order $64$ and we have shown that  if $\KK$ is an unramified quadratic extension of $\kk$ other than $\KK_3=\kk(\sqrt 2)$, then  $\mathbf{C}l_2(\KK)$ is of type $(2, 4)$ or $(2, 2, 2)$. In this paper we complete  this study by determining the structure of $G$, the abelian type invariants of the 2-class groups of all the unramified extensions of $\kk$ within $\kk_2^{(1)}$ and the  kernel  of the natural class extension homomorphism $j_{\KK/\kk}: Cl_2(\kk) \longrightarrow Cl_2(\KK)$, where  $\KK$ is an unramified  extension of $\kk$ within $\kk_2^{(1)}$. The  main results of this paper are  Theorems  \ref{2} and \ref{3} below; whereas   Theorem \ref{1}  is proved in \cite{AZT12-1},    \cite{AT08} and \cite{Ka76}.
\subsection{Unramified extensions of $\kk$}
The first and the second assertions  of the following theorem hold according to \cite{Ka76} and \cite{AT08} respectively, the others are shown in \cite{AZT12-1}.
\begin{them}\label{1}
Let $p_1$, $p_2$ be as above.
\begin{enumerate}[\rm\indent(1)]
  \item The $2$-class groups   of $k_0$, $\overline{k}_0$ are of type $(2, 2)$.
  \item The $2$-class group,  $\mathbf{C}l_2(\kk)$, of $\kk$ is of type $(2, 2, 2)$.
  \item The discriminant of $\kk$ is: $disc(\kk)=disc(k).disc(k_0).disc(\overline{k}_0)=2^8p_1^2p_2^2.$
  \item $\kk$ has seven unramified quadratic extensions within its Hilbert $2$-class field  $\kk_2^{(1)}$. They are given by:
\begin{center}$\KK_1=\kk(\sqrt{p_1})$,\qquad $\KK_2=\kk(\sqrt{p_2})$,\qquad  $\KK_3=\kk(\sqrt 2)$,\\
 $\KK_4=\kk(\sqrt{\pi_1\pi_3})$, \ $\KK_5=\kk(\sqrt{\pi_1\pi_4})$,\ $\KK_6=\kk(\sqrt{\pi_2\pi_3})$ and  $\KK_7=\kk(\sqrt{\pi_2\pi_4})$.
  \end{center}
  \item $\KK_1$, $\KK_2$, $\KK_3$ are intermediate fields between $\kk$ and its genus field $\k$. The fields $\KK_4\simeq\KK_7$
and $\KK_5 \simeq \KK_6$ are pairwise conjugate and thus isomorphic. Consequently $\KK_1$, $\KK_2$, $\KK_3$ are
absolutely abelian, whereas $\KK_4$, $\KK_5$, $\KK_6$, $\KK_7$ are  non-normal over $\QQ$.
  \item $\kk$ has seven unramified bicyclic biquadratic extensions within its Hilbert $2$-class field  $\kk_2^{(1)}$. One of them is
       \begin{center}$\mathbb{L}_1=\KK_1.\KK_2.\KK_3=\k=\QQ(\sqrt{p_1}, \sqrt{p_2}, \sqrt{q}, \sqrt{-1})$,\end{center}
the absolute genus field of $\kk$ and the others are given by:
\begin{center}
$\mathbb{L}_2=\KK_1.\KK_4.\KK_6,\quad
\mathbb{L}_3=\KK_1.\KK_5.\KK_7$,\quad
$\mathbb{L}_4=\KK_2.\KK_4.\KK_5\ and\
\mathbb{L}_5=\KK_2.\KK_6.\KK_7$
\end{center}
are non-normal over $\QQ$;  moreover $\LL_2\simeq \LL_3$ and  $\LL_4 \simeq\LL_5$.
 \begin{center} $\mathbb{L}_6=\KK_3.\KK_4.\KK_7$ \text{ and }
$\mathbb{L}_7=\KK_3.\KK_5.\KK_6$ \text{ are absolutely Galois. }
  \end{center}
\end{enumerate}
\end{them}
\subsection{Structure of $G=\mathrm{Gal}(\L/\kk)$}
Let $\mathcal{H}_0$ (resp. $\mathcal{H}_1$, $\mathcal{H}_2$) denote the prime ideal of $\kk$ lying above $1+i$ (resp. $\pi_1$, $\pi_2$). Write $q=q(\KK_3^+/\QQ)$ for simplicity.
\begin{them}\label{2}
Keep the preceding assumptions. Then
\begin{enumerate}[\upshape\indent(1)]
  \item  $\mathbf{C}l_2(\kk)=\langle[\mathcal{H}_0],  [\mathcal{H}_1], [\mathcal{H}_2]\rangle\simeq(2, 2, 2)$.
  \item $\mathbf{C}l_2(\KK_3)$ is of type
   $\left\{
   \begin{array}{ll}
   (2^{m}, 2^{n+1})  \text{ if } q=1,\\
   (2^{\min(m, n+1)}, 2^{\max(m+1, n+2)})  \text{ if } q=2,
   \end{array}\right.$\\
   where $n$ and $m$ are determined by $2^{m+1}=h(-p_1p_2)$, $m\geq2$, and $2^{n}=h(p_1p_2)$, $n\geq1$; and either $n\geq3$ or $m\geq3$.
  \item The length of the $2$-class field tower of $\kk$ is $2$.
  \item $G=\mathrm{G}al(\kk_2^{(2)}/\kk)$ is given by:
  \begin{enumerate}[\rm\indent(i)]
    \item If $q=1$, then
\begin{align*}
G=\langle\  \rho, \tau, \sigma:& \    \rho^4=\sigma^{2^{m}}=\tau^{2^{n+1}}=1,\ \rho^2=\psi,\  [\tau, \sigma]=1,\\
                                      & [\rho, \sigma]=\ \sigma^{2},\  [\rho, \tau]=\tau^2\ \rangle,\text{ where }
                                     \end{align*}
                      $ \psi=\left\{ \begin{array}{ll}
                          \sigma^{2^{m-1}}\text{ if } (\frac{p_1}{p_2})=1 \text{ and } N(\varepsilon_{p_1p_2})=1,\\
                          \tau^{2^n}\sigma^{2^{m-1}}\text{ if } (\frac{p_1}{p_2})=-1\text{ or }(\frac{p_1}{p_2})=1 \text{ and } N(\varepsilon_{p_1p_2})=-1.\end{array}\right.$
    \item If $q=2$, then
\begin{align*}
G=\langle\ \rho,& \tau, \sigma:\  \rho^4=\sigma^{2^{m+1}}=\tau^{2^{n+2}}=1,\ \sigma^{2^{m}}=\tau^{2^{n+1}},\\
 \rho^2&=\tau^{2^n}\sigma^{2^{m-1}}, [\tau, \sigma]=1,\  [\rho, \sigma]=\sigma^{-2},\  [\rho, \tau]=\tau^2 \rangle.
 \end{align*}
  \end{enumerate}
\item The derived group of $G$ is $G'=\langle \sigma^{2}, \tau^2 \rangle$ and $\mathbf{C}l_2(\kk^{(1)}_2)$ is of type\\
   $\left\{
   \begin{array}{ll}
   (2^{m-1}, 2^{n})  \text{ if } q=1,\\
   (2^{\min(m, n+1)-1}, 2^{\max(m+1, n+2)-1})=(2, 2^{n+1})  \text{ if } q=2.
   \end{array}\right.$
\item The  coclass of G is $3$ and its nilpotency class is\\
$\left\{ \begin{array}{ll}
\max(n, m-1)+1 \text { if } q=1,\\
\max(n+1, m)+1  \text { if } q=2.
\end{array}\right.$
\end{enumerate}
\end{them}
\subsection{Abelian type invariants and capitulation kernels}
Let $N_j$ denote the  subgroup   $N_{\KK_j/\kk}(\mathbf{C}l_2(\KK_j))$ of $\mathbf{C}l_2(\kk)$ and $\kappa_{\KK}$ denote  the kernel of the natural class
extension homomorphism  $j_{\KK/\kk}: \mathbf{C}l_2(\kk)\longrightarrow \mathbf{C}l_2(\KK)$, where $\KK$ is an unramified extension of $\kk$ within $\kk_2^{(1)}$.
\begin{them}\label{3}
Let $2^n=h(p_1p_2)$, $2^{m+1}=h(-p_1p_2)$, where $n\geq1$ and  $m\geq2$.
\begin{enumerate}[\rm\indent(1)]
\item  $\#\kappa_{\KK_j}=4$, for all $j\neq3$. If $j=3$, then
     $\#\kappa_{\KK_3}=\left\{\begin{array}{ll}
     4 \text{ if } q=1,\\
     2 \text{ if } q=2.
     \end{array}
    \right.$
\item  All the extensions $\KK_j$ satisfy  Taussky's condition $(A)$ i.e. $\#\kappa_{\KK_j}\cap N_j>1$, for details see  \cite{Ta-70}.
 \item  The order of  $\kappa_{\LL_j}$ is $8$ (total $2$-capitulation), for all $j$, and $\LL_j$ are of type $(A)$.
  \item The abelian type invariants of the $2$-class groups $\mathbf{C}l_2(\KK_j)$ are given by:
\begin{enumerate}[\rm\indent(i)]
  \item $\mathbf{C}l_2(\KK_1)\simeq\mathbf{C}l_2(\KK_2)\simeq
   \left\{\begin{array}{ll}
      (2, 2, 2) \text{ if  } \left(\frac{p_1}{p_2}\right)=1,\\
       (2, 4) \text{ otherwise}.
       \end{array}\right.$
  \item If $(\frac{p_1}{p_2})=1$,
     then $\mathbf{C}l_2(\KK_4)$,  $\mathbf{C}l_2(\KK_5)$, $\mathbf{C}l_2(\KK_6)$ and  $\mathbf{C}l_2(\KK_7)$ are of type $(2, 2, 2)$ if $(\frac{\pi_1}{\pi_3})=-1$, and  of type $(2, 4)$ otherwise.
  \item Assume $(\frac{p_1}{p_2})=-1$.\\
     If $(\frac{\pi_1}{\pi_3})=-1$, then $\left\{\begin{array}{ll} \mathbf{C}l_2(\KK_4)\simeq\mathbf{C}l_2(\KK_7)\simeq(2, 4),\\  \mathbf{C}l_2(\KK_5)\simeq\mathbf{C}l_2(\KK_6)\simeq(2, 2, 2).\end{array}\right.$\\
     If $(\frac{\pi_1}{\pi_3})=1$, then   $\left\{\begin{array}{ll} \mathbf{C}l_2(\KK_4)\simeq\mathbf{C}l_2(\KK_7)\simeq(2, 2, 2),\\  \mathbf{C}l_2(\KK_5)\simeq\mathbf{C}l_2(\KK_6)\simeq(2, 4).\end{array}\right.$
\end{enumerate}
  \item The abelian type invariants of the $2$-class groups $\mathbf{C}l_2(\LL_j)$ are given by:
\begin{enumerate}[\rm\indent(i)]
\item  $\mathbf{C}l_2(\LL_1)=\mathbf{C}l_2(\k)\simeq
   \left\{\begin{array}{ll}
(2^m, 2^{n}) \text{ if } q=1,\\
(2^{\min(m,n)}, 2^{\max(m+1,n+1)})\hbox{ if } q=2.
      \end{array}
    \right.$
\item If $(\frac{p_1}{p_2})=-1$ or $(\frac{p_1}{p_2})=(\frac{\pi_1}{\pi_3})=1$, then $\mathbf{C}l_2(\LL_2)$, $\mathbf{C}l_2(\LL_3)$, $\mathbf{C}l_2(\LL_4)$ and $\mathbf{C}l_2(\LL_5)$ are of type $(2, 4)$.\\
   If $(\frac{p_1}{p_2})=-(\frac{\pi_1}{\pi_3})=1$, then  $\mathbf{C}l_2(\LL_2)$, $\mathbf{C}l_2(\LL_3)$, $\mathbf{C}l_2(\LL_4)$ and $\mathbf{C}l_2(\LL_5)$ are of type $(2, 2, 2)$.
\item $(a)$ Assume $q=2$, so  $\mathbf{C}l_2(\LL_6)$ and  $\mathbf{C}l_2(\LL_7)$ are of type $(2, 2^{n+2})$ if  $(\frac{p_1}{p_2})=1$,
 otherwise we have:\\
   $\mathbf{C}l_2(\LL_6)\simeq\left\{ \begin{array}{ll}
 (4, 4)  \text{ if } (\frac{\pi_1}{\pi_3})=1,\\
   (2, 8) \text{ if } (\frac{\pi_1}{\pi_3})=-1,
   \end{array}\right.$
   $\mathbf{C}l_2(\LL_7)\simeq\left\{ \begin{array}{ll}
 (2, 8)  \text{ if } (\frac{\pi_1}{\pi_3})=1,\\
   (4, 4) \text{ if } (\frac{\pi_1}{\pi_3})=-1.
   \end{array}\right.$
 $(b)$ Assume $q=1$.\\
   $\text{ If } (\frac{1+i}{\pi_1})(\frac{1+i}{\pi_3})=1,\text{ then }
   \left\{
 \begin{array}{ll}\mathbf{C}l_2(\LL_6)\simeq(2^{m-1}, 2^{n+1}),\\
   \mathbf{C}l_2(\LL_7)\simeq(2^{\min(m-1, n)}, 2^{\max(m, n+1)}).
   \end{array} \right.$\\
$\text{ If } (\frac{1+i}{\pi_1})(\frac{1+i}{\pi_3})=-1,\text{ then }
   \left\{
   \begin{array}{ll}
   \mathbf{C}l_2(\LL_6)\simeq(2^{\min(m-1, n)}, 2^{\max(m, n+1)}),\\
    \mathbf{C}l_2(\LL_7)\simeq (2^{m-1}, 2^{n+1}).
     \end{array} \right.$
\end{enumerate}
\end{enumerate}
\end{them}
\section{Preliminary results}
 Let $p_1$, $p_2$  be different primes satisfying the conditions mentioned at the beginning of $\S2$, and set  $k_1=\QQ(\sqrt{p_1p_2})$,  $\overline{k}_1=\QQ(\sqrt{-p_1p_2})$. Put  $\varepsilon_{2p_1p_2}=x+y\sqrt{2p_1p_2}$ and $\varepsilon_{p_1p_2}=a+b\sqrt{p_1p_2}$. Let $\displaystyle\left(\frac{g, h}{p}\right)$ denote the quadratic Hilbert  symbol for the prime $p$.
\begin{lem}\label{8}  Let $\varepsilon_d$ denote the fundamental unit of $k_0$. Then
\begin{enumerate}[\rm\indent(i)]
  \item $N(\varepsilon_d)=-1$.
  \item If $N(\varepsilon_{p_1p_2})=1$, then $2p_1(a\pm1)$ (i.e. $2p_2(a\mp1)$) is a square in $\NN$.
\end{enumerate}
\end{lem}
\begin{proof}
(i) See \cite{AT08}.\\
(ii) As $N(\varepsilon_{p_1p_2})=1$, so $\left(\frac{p_1}{p_2}\right)=1$ and $a^2-1=b^2p_1p_2$, so:\\
(a) If $a\pm1$ is a square in $\NN$, then $\left(\frac{2}{p_1}\right)=-1$, which is false.\\
(b) If $p_1(a\pm1)$ is a square in $\NN$, then $\left(\frac{p_1}{p_2}\right)=\left(\frac{2}{p_1}\right)=-1$, which is absurd. And the result derived.
\end{proof}
 \begin{lem}\label{4}
If $\left(\frac{p_1}{p_2}\right)=1$, then $\left(\frac{p_1}{p_2}\right)_4\left(\frac{p_2}{p_1}\right)_4=\left(\frac{\pi_1}{\pi_3}\right)$.
 \end{lem}
 \begin{proof}
 From \cite{Ka76} we get $\left(\frac{p_1}{p_2}\right)_4\left(\frac{p_2}{p_1}\right)_4=\left(\frac{p_1}{ac+bd}\right)$, where $p_1=a^2+b^2$ and $p_2=c^2+d^2$; on the other hand, according to  \cite{Lm00} we have  $\left(\frac{p_1}{ac+bd}\right)=\left(\frac{\pi_1}{\pi_3}\right)$, which implies the result.
 \end{proof}
 \begin{lem}\label{6}
Put $\kk=\QQ(\sqrt{2p_1p_2}, i)$ and $\KK_3=\QQ(\sqrt{2}, \sqrt{p_1p_2}, i)$. Then
\begin{enumerate}[\rm\indent(1)]
  \item $\{\varepsilon_{2p_1p_2}\}$ is a $\mathrm{FSU}$ of $\kk$.
  \item If $N(\varepsilon_{p_1p_2})=1$ or $N(\varepsilon_{p_1p_2})=-1$ and $\sqrt{\varepsilon_{2}\varepsilon_{p_1p_2}\varepsilon_{2p_1p_2}}\not\in \KK_3^+$, then
 \begin{enumerate}[\rm\indent(i)]
   \item  $\{\varepsilon_2, \varepsilon_{p_1p_2}, \varepsilon_{2p_1p_2}\}$ is a $\mathrm{FSU}$ of both $\KK_3^+$ and $\KK_3$.
   \item  $q=1$,  $q(\KK_3/\QQ)=2$ and $h(\KK_3)=h(p_1p_2)h(-p_1p_2)$.
 \end{enumerate}
  \item If $N(\varepsilon_{p_1p_2})=-1$ and $\sqrt{\varepsilon_{2}\varepsilon_{p_1p_2}\varepsilon_{2p_1p_2}}\in \KK_3^+$, then
 \begin{enumerate}[\rm\indent(i)]
   \item  $\{\varepsilon_2, \varepsilon_{p_1p_2},  \sqrt{\varepsilon_2\varepsilon_{p_1p_2}\varepsilon_{2p_1p_2}}\}$ is a $\mathrm{FSU}$ of both $\KK_3^+$ and $\KK_3$.
   \item  $q=2$,  $q(\KK_3/\QQ)=4$  and $h(\KK_3)=2h(p_1p_2)h(-p_1p_2)$.
 \end{enumerate}
\end{enumerate}
\end{lem}
 \begin{proof}
(1) As $N(\varepsilon_{2p_1p_2})=-1$, so if $N(\varepsilon_{p_1p_2})=-1$  we get, according to \cite[Applications 1), p.114]{Az-05}, that $\{\varepsilon_{2p_1p_2}\}$ is a $\mathrm{FSU}$ of $\kk$.\\
\indent (2) Assume that $N(\varepsilon_{p_1p_2})=1$. As $N(\varepsilon_{2})=N(\varepsilon_{2p_1p_2})=-1$, so $\varepsilon_{2}$, $\varepsilon_{2p_1p_2}$, $\varepsilon_{2}\varepsilon_{p_1p_2}$, $\varepsilon_{2}\varepsilon_{2p_1p_2}$, $\varepsilon_{p_1p_2}\varepsilon_{2p_1p_2}$ and $\varepsilon_{2}\varepsilon_{p_1p_2}\varepsilon_{2p_1p_2}$  are not squares in $\KK_3^+$, else by taking a suitable norm we get $i\in \KK_3^+$, which is false. Furthermore  $(2+\sqrt2)\varepsilon_{2}^i\varepsilon_{p_1p_2}^j\varepsilon_{2p_1p_2}^k$ can not be a square
in $\KK_3^+$, for all i, j and k in $\{0, 1\}$, as otherwise with some $\alpha\in\KK_3^+$ we would have $\alpha^2=(2+\sqrt 2)\varepsilon_{2}^i\varepsilon_{p_1p_2}^j\varepsilon_{2p_1p_2}^k$, so $(N_{\KK^+/\QQ(\sqrt{p_1p_2})}(\alpha))^2=2(-1)^{i+k}\varepsilon_{p_1p_2}^{2j}$, yielding that $\sqrt2\in\QQ(\sqrt{p_1p_2})$, which is absurd.

 Put $\varepsilon_{p_1p_2}=a+b\sqrt{p_1p_2}$; as $2p_1(a\pm1)$ is a square in $\NN$, thus $\sqrt{2\varepsilon_{p_1p_2}}=b_1\sqrt{2p_1}+b_2\sqrt{2p_2}$, where $b_i\in\ZZ$; so $\varepsilon_{p_1p_2}$ is not a square in  $\KK_3^+$; hence $\{\varepsilon_2, \varepsilon_{p_1p_2}, \varepsilon_{2p_1p_2}\}$ is a $\mathrm{FSU}$ of $\KK_3^+$, which implies that $q=1$. Thus from  \cite[Proposition 3, p.112]{Az-05} we get   $\{\varepsilon_2, \varepsilon_{p_1p_2}, \varepsilon_{p_1p_2q}\}$ is a $\mathrm{FSU}$ of
 $\KK_3$, we infer that  $q(\KK_3/\QQ)=2$, since $\sqrt i\in\KK_3$.\\
 If $N(\varepsilon_{p_1p_2})=-1$, then the results are guaranteed by  \cite[Propositions 8, 15]{Az-05}.
  In the end,  under our conditions, P.Kaplan states in \cite{Ka76} that $h(2p_1p_2)=h(-2p_1p_2)=4$, therefore the class number formula yields that $h(\KK_3^+)=h(p_1p_2)$ and $h(\KK_3)=h(p_1p_2)h(-p_1p_2)$.\\
 \indent  (3) If  $N(\varepsilon_{p_1p_2})=-1$ and $\sqrt{\varepsilon_{2}\varepsilon_{p_1p_2}\varepsilon_{2p_1p_2}}\in \KK_3^+$, then the results are also deduced from \cite[Propositions 8, 15]{Az-05} and the class number formula implies $(3)(iii)$.
 \end{proof}
 \begin{lem}\label{7}
 Let $\kappa_{\KK_3}$ denote the set of classes of $\mathbf{C}l_2(\kk)$ that capitulate in $\KK_3$, then
 $\kappa_{\KK_3}=\left\{
 \begin{array}{ll}
\langle[\mathcal{H}_0]\rangle \text{ if } q=2,\\
\langle[\mathcal{H}_0], [\mathcal{H}_1\mathcal{H}_2]\rangle \text{ if }  q=1.
 \end{array}\right.$
 \end{lem}
\begin{proof} From Lemma \ref{6} we get $E_\kk=\langle i, \varepsilon_{2p_1p_2}\rangle$ and
$E_{\KK_3}=\langle \sqrt i,  \varepsilon_{2}, \varepsilon_{p_1p_2}, \varepsilon_{2p_1p_2}\rangle$ or\\
$E_{\KK_3}=\langle \sqrt i, \varepsilon_{2}, \varepsilon_{p_1p_2}, \sqrt{\varepsilon_2\varepsilon_{p_1p_2}\varepsilon_{p_1p_2q}}\rangle$,
    according as $q=1$ or $2$. Therefore  $N_{\KK_3/\kk}(E_{\KK_3})=\langle i, \varepsilon_{2p_1p_2}^2\rangle$ or $\langle i, \varepsilon_{2p_1p_2}\rangle$, thus
 \begin{center}$[E_\kk: N_{\KK_3/\kk}(E_{\KK_3})]=\left\{
 \begin{array}{ll}
 2, \text{ if } q=1,\\
 1, \text{ if }  q=2;
 \end{array}\right.$
 hence $\#\kappa_{\KK_3}=\left\{
 \begin{array}{ll}
 4,\text{ if } q=1,\\
 2, \text{ if }  q=2.
 \end{array}\right.$ \end{center}
Moreover, it is easy to see that   $\sqrt{(1+i)\varepsilon_2}=\frac{1}{2}(2+(1+i)\sqrt2)$, so there exists $\beta\in\KK_3$ such that $\mathcal{H}_0^2=(1+i)=(\beta^2)$, this implies that $\mathcal{H}_0$ capitulates in $\KK_3$. Consequently,
if  $q=2$, then $\kappa_{\KK_3}=\langle[\mathcal{H}_0]\rangle$.\\
\indent Suppose   $q=1$, then we have two cases to discuss:\\
(a) If $N(\varepsilon_{p_1p_2})=1$, then by  Lemma \ref{8}, we get $\sqrt{2\varepsilon_{p_1p_2}}=b_1\sqrt{2p_1}+b_2\sqrt{2p_2}$, where $b=2b_1b_2$, from which we deduce that $2p_1\varepsilon_{p_1p_2}$ is a square in $\KK_3$, thus there exists  $\alpha\in\KK_3$ such that $(2p_1)=(\alpha^2)$. On the other hand, $(\mathcal{H}_1\mathcal{H}_2)^2=(p_1)$ and $(2)=(2i)=(1+i)^2$, hence $\mathcal{H}_1\mathcal{H}_2=(\frac{\alpha}{1+i})$, which implies that $\mathcal{H}_1\mathcal{H}_2$ capitulates in $\KK_3$.\\
(b) If $N(\varepsilon_{p_1p_2})=-1$, then there exist an even integer $a$
and an odd integer  $b$  such that $\varepsilon_{p_1p_2}=a+b\sqrt{p_1p_2}$, so $a^2+1=b^2p_1p_2$,
 since $p_1p_2\equiv1 \pmod8$.
Therefore:
 \begin{equation*}
\left\{
 \begin{array}{ll}
 a\mp i=ib_1^2\pi_1\pi_3,\\
 a\pm i=-ib_2^2\pi_2\pi_4,
 \end{array}\right. \text{ or }
 \left\{
 \begin{array}{ll}
 a\mp i=ib_1^2\pi_1\pi_4,\\
 a\pm i=-ib_2^2\pi_2\pi_3,
 \end{array}\right.
 \end{equation*}
 \begin{equation}\label{18}
 \text{ hence }
 \left.
 \begin{array}{ll}\sqrt{\varepsilon_{p_1p_2}}=z_1\sqrt{\pi_1\pi_3}+z_2\sqrt{\pi_2\pi_4}
 \text{\   or } \\ \sqrt{\varepsilon_{p_1p_2}}=z_1\sqrt{\pi_1\pi_4}+z_2\sqrt{\pi_2\pi_3}, \end{array} \right\}\end{equation}
 where $z_2$ is the conjugate of $z_1$ in $\frac{1}{2}\ZZ[i]$.\\
 \indent Similarly, as $N(\varepsilon_{2p_1p_2})=-1$, so there exist $x$, $y$ in $\NN$ such that $x^2+1=2p_1p_2y^2$, and
     \begin{equation}\label{19}
    \left.
   \begin{aligned}
  \sqrt{\varepsilon_{2p_1p_2}}& = y_1\sqrt{(1+i)\pi_1\pi_3}+y_2\sqrt{(1-i)\pi_2\pi_4}),\ or  \\
  \sqrt{\varepsilon_{2p_1p_2}}& = y_1\sqrt{(1+i)\pi_1\pi_4}+y_2\sqrt{(1-i)\pi_2\pi_3}),\ or
    \  \\
    \sqrt{2\varepsilon_{2p_1p_2}}& = y_1\sqrt{(1+i)\pi_1\pi_3}+y_2\sqrt{(1-i)\pi_2\pi_4}),\ or \\
    \sqrt{2\varepsilon_{2p_1p_2}}& = y_1\sqrt{(1+i)\pi_1\pi_3}+y_2\sqrt{(1-i)\pi_2\pi_4}),
     \end{aligned}
   \right\}
\end{equation}
where $y_i$ are in  $\ZZ[i]$ or $\frac{1}{2}\ZZ[i]$.\\
 Finally,  Note  that:
 \begin{equation}\label{20}
  \sqrt{2\varepsilon_1}=\sqrt{1+i}+\sqrt{1-i}.\end{equation}
So by multiplying  the equalities  $(\ref{18})$, $(\ref{19})$ and $(\ref{20})$ we get
\begin{center}
$\sqrt{\varepsilon_1\varepsilon_2\varepsilon_3}=\alpha+\beta\sqrt{2}+\gamma\sqrt{p_1p_2}+\delta\sqrt{2p_1p_2}
\in\QQ(\sqrt2, \sqrt{p_1p_2})$ or\\   $\sqrt{\varepsilon_1\varepsilon_2\varepsilon_3}=\alpha\sqrt{p_1}+
\beta\sqrt{p_2}+\gamma\sqrt{2p_1}+\delta\sqrt{2p_2}\not\in\QQ(\sqrt2, \sqrt{p_1p_2})$,\end{center} where
$\alpha$, $\beta$, $\gamma$ and $\delta$ are in $\QQ$.\\
 As $q=1$, so $\varepsilon_2\varepsilon_{p_1p_2}\varepsilon_{p_1p_2q}$ is not a square in $\KK_3^+$, hence  $p_1\varepsilon_2\varepsilon_{p_1p_2}\varepsilon_{p_1p_2q}$ is a square in $\KK_3$; which yields that   $\mathcal{H}_1\mathcal{H}_2$ capitulates in $\KK_3$. Thus
$\kappa_{\KK_3}=\langle[\mathcal{H}_0], [\mathcal{H}_1\mathcal{H}_2]\rangle$.
 \end{proof}
  \begin{propo}[\cite{AZT-3}]\label{5}
Let $p_1\equiv p_2\equiv1 \pmod4$  be  different primes such that\\  $\displaystyle\left(\frac{2}{p_1}\right)=\displaystyle\left(\frac{2}{p_2}\right)=
\displaystyle\left(\frac{p_1}{p_2}\right)$. Then
$$\displaystyle\left(\frac{p_1p_2}{2}\right)_4\displaystyle\left(\frac{2p_1}{p_2}\right)_4
\displaystyle\left(\frac{2p_2}{p_1}\right)_4=\displaystyle\left(\frac{\pi_1}{\pi_3}\right)
\displaystyle\left(\frac{1+i}{\pi_1}\right)\displaystyle\left(\frac{1+i}{\pi_3}\right).$$
\end{propo}
 \begin{propo}[\cite{AZT-3}]\label{11}
 Let  $p_1\equiv p_2\equiv1 \pmod4$  be different  primes such that \\ $\displaystyle\left(\frac{2}{p_1}\right)=\displaystyle\left(\frac{2}{p_2}\right)=\displaystyle\left(\frac{p_1}{p_2}\right)=-1$. Then the following  assertions are equivalent:
\begin{enumerate}[\upshape\indent(1)]
  \item $\varepsilon_2\varepsilon_{p_1p_2}\varepsilon_{2p_1p_2}$ is a square in $\KK_3^+$.
  \item $\displaystyle\left(\frac{p_1p_2}{2}\right)_4\displaystyle\left(\frac{2p_1}{p_2}\right)_4
\displaystyle\left(\frac{2p_2}{p_1}\right)_4=-1.$
  \item $q(\KK_3^+/\QQ)=2$.
\end{enumerate}
\end{propo}
The following results are deduced from \cite{Lm94}.
  \begin{them}\label{9}
  Let $p_1\equiv p_2\equiv5 \pmod8$ be different primes and put  $F_1=\QQ(\sqrt{p_1p_2}, i)$.
  \begin{enumerate}[\rm\indent(1)]
    \item   $\mathbf{C}l_2(\overline{k}_1)$ is of type $(2, 2^m)$,  $m\geq2$. It is generated  by  $\mathfrak{2}=(2, 1+\sqrt{-p_1p_2})$, the prime ideal of  $\overline{k}_1$ above $2$, and an ideal  $I$ of $\overline{k}_1$ of order $2^m$. Moreover
      \begin{center} $\left\{
       \begin{array}{ll}
       I^{2^{m-1}}\sim\mathfrak{p}_1 \text{ if } \left(\dfrac{p_1}{p_2}\right)=1,\\
       I^{2^{m-1}}\sim\mathfrak{2}\mathfrak{p}_1 \text{ if } \left(\dfrac{p_1}{p_2}\right)=-1;
       \end{array}\right.$\end{center}
        where $\mathfrak{p}_1=(p_1, \sqrt{-p_1p_2})$ is the prime ideal of $\overline{k}_1$ above $p_1$.
    \item  $\mathbf{C}l_2(k_1)$ is of type $(2^n)$,  $n\geq1$, and it is generated by $\mathfrak{2}_1$, a prime ideal of $k_1$ above $2$.
    \item  $\mathbf{C}l_2(F_1)$ is of $2$-rank equal to $2$. It is generated by $I$ and $\mathfrak{2}_{F_1}$, where  $\mathfrak{2}_{F_1}$ is a prime ideal of $F_1$ above $2$.
    \item If $\left(\frac{p_1}{p_2}\right)=-1$, then $\mathbf{C}l_2(F_1)\simeq(2^{n}, 2^{m})$; and, in $\mathbf{C}l_2(F_1)$, $I^{2^{m-1}}\sim \mathfrak{2}_{F_1}^{2^n}\sim \mathfrak{p}_1\not\sim1$.
    \item If $\left(\frac{p_1}{p_2}\right)=1$ and $N(\varepsilon_{p_1p_2})=-1$, then $$\mathbf{C}l_2(F_1)\simeq(2^{\min(n, m-1)}, 2^{\max(m-1, n+1)})$$ and  $I^{2^{m-1}}\sim \mathfrak{2}_{F_1}^{2^n}\sim \mathfrak{p}_1\not\sim1$.
      \item If  $\left(\frac{p_1}{p_2}\right)=1$ and $N(\varepsilon_{p_1p_2})=1$, then $\mathbf{C}l_2(F_1)\simeq(2^{n+1}, 2^{m-1})$; moreover  $I^{2^{m-1}}\sim \mathfrak{2}_{F_1}^{2^{n+1}}\sim \mathfrak{p}_1\sim1$.
   \end{enumerate}
 \end{them}
 Using  the above theorem, we prove the following lemma.
   \begin{lem}\label{16}
   Let  $\mathfrak{p}_1\mathcal{O}_{F_1}=\mathcal{P}_1\mathcal{P}_2$ and  $p_2\mathcal{O}_{F_1}=\mathcal{P}_3^2\mathcal{P}_4^2$, then in $\mathbf{C}l_2(F_1)$ we have:
     \begin{enumerate}[\rm\indent(i)]
       \item If $\left(\frac{p_1}{p_2}\right)=-1$ or $\left(\frac{p_1}{p_2}\right)=1$ and $N(\varepsilon_{p_1p_2})=-1$, then  $\mathcal{P}_1\sim \mathfrak{2}_{F_1}^{2^{n-1}}I^{2^{m-2}}$.
      \item If  $\left(\frac{p_1}{p_2}\right)=1$ and $N(\varepsilon_{p_1p_2})=1$, then  $\mathcal{P}_1\sim \mathfrak{2}_{F_1}^{2^{n}}I^{2^{m-2}}$ or  $\mathcal{P}_1\sim I^{2^{m-2}}$. Moreover $\mathcal{P}_1\mathcal{P}_3\sim \mathfrak{2}_{F_1}^{2^{n}}$.
   \end{enumerate}
    \end{lem}
\begin{proof}
Let $p_1\mathcal{O}_{\QQ(i)}=\pi_1\pi_2$, $p_2\mathcal{O}_{\QQ(i)}=\pi_3\pi_4$, $\mathfrak{p}_1\mathcal{O}_{F_1}=\mathcal{P}_1\mathcal{P}_2$ and $\mathfrak{p}_2\mathcal{O}_{F_1}=\mathcal{P}_3\mathcal{P}_4$, where $\mathfrak{p}_2$ is the prime ideal of $\overline{k}_1$ above $p_2$, then $(\pi_i)=\mathcal{P}_i^2$, for all $i$. So, according to \cite[Proposition 1]{AZT12-2}, $\mathcal{P}_i$ are not principals in $F_1$ and they are of order two. On the other hand, as the 2-rank of $\mathbf{C}l_2(F_1)$ is 2, thus  $\mathcal{P}_i \in\langle[\mathfrak{2}_{F_1}], [I]\rangle$.\\
 \indent (i) In this case,  we have $\mathfrak{p}_1\not\sim1$, hence  $\mathcal{P}_1\not\sim\mathcal{P}_2$; note that the elements of order two in $\mathbf{C}l_2(F_1)$ are $\mathfrak{2}_{F_1}^{2^{n-1}}I^{2^{m-2}}$, $\mathfrak{2}_{F_1}^{2^{n-1}}I^{-2^{m-2}}$ and $\mathfrak{2}_{F_1}^{2^{n}}\sim I^{2^{m-1}}$. Therefore $\mathcal{P}_1$ is equivalent to one of these three elements.  As $\mathcal{P}_1\sim\mathfrak{2}_{F_1}^{2^{n}}\sim I^{2^{m-1}}$ can not occur, if not we would have, by applying  the norm $N_{F_1/\overline{k}_1}$, $\mathfrak{p}_1\sim I^{2^{m}}\sim1$, which is false. Thus $\mathcal{P}_1\sim\mathfrak{2}_{F_1}^{2^{n-1}}I^{2^{m-2}}$ and $\mathcal{P}_2\sim\mathfrak{2}_{F_1}^{2^{n-1}}I^{-2^{m-2}}$ or $\mathcal{P}_1\sim\mathfrak{2}_{F_1}^{2^{n-1}}I^{-2^{m-2}}$ and  $\mathcal{P}_2\sim\mathfrak{2}_{F_1}^{2^{n-1}}I^{2^{m-2}}$. Hence with out loss of generality we can choose $\mathcal{P}_1\sim\mathfrak{2}_{F_1}^{2^{n-1}}I^{2^{m-2}}$.\\
  \indent (ii) In this case,  we have $\mathfrak{p}_1\sim\mathfrak{p}_2\sim1$, hence  $\mathcal{P}_1\sim\mathcal{P}_2$ and $\mathcal{P}_3\sim\mathcal{P}_4$. On  the other hand, according to \cite[Proposition 1]{AZT12-2}, $\mathcal{P}_1\mathcal{P}_3$ is not principal in $F_1$. To this end, note that the elements of order two in $\mathbf{C}l_2(F_1)$ are $\mathfrak{2}_{F_1}^{2^{n}}I^{2^{m-2}}$, $\mathfrak{2}_{F_1}^{2^{n}}$ and $I^{-2^{m-2}}$. Therefore $\mathcal{P}_1$ is equivalent to one of these three elements. As $\mathcal{P}_1\sim\mathfrak{2}_{F_1}^{2^{n}}$ can not occur, as otherwise, by applying  the norm $N_{F_1/\overline{k}_1}$, we get $\mathfrak{p}_1\sim \mathfrak{2}^{2^{n}}\sim1$, which is false. Thus $\mathcal{P}_1\sim I^{2^{m-2}}$ and $\mathcal{P}_3\sim\mathfrak{2}_{F_1}^{2^{n}}I^{2^{m-2}}$ or $\mathcal{P}_1\sim\mathfrak{2}_{F_1}^{2^{n-1}}I^{2^{m-2}}$ and  $\mathcal{P}_3\sim I^{2^{m-2}}$. Hence $\mathcal{P}_1\mathcal{P}_3\sim \mathfrak{2}_{F_1}^{2^{n}}$.
\end{proof}
We conclude this section  with the following lemma which gives the relationship between the unit index  $q$ and the integers $n$ and $m$. It is a consequence of Lemma \ref{6},  Proposition \ref{11} and  the results in \cite{Sc-34}, \cite{Ka76}.
  \begin{lem}\label{12}
 $(1)$ Suppose $q=1$, so
 \begin{enumerate}[\rm\indent(i)]
   \item If $\left(\frac{p_1}{p_2}\right)=-1$, then $n=1$ and $m\geq3$.
   \item If $\left(\frac{p_1}{p_2}\right)=1$, then
  \begin{enumerate}[\rm\indent(a)]
   \item If $\left(\frac{p_1}{p_2}\right)_4\left(\frac{p_2}{p_1}\right)_4=-1$, then $n=1$ and $m\geq3$.
   \item If $\left(\frac{p_1}{p_2}\right)_4\left(\frac{p_2}{p_1}\right)_4=1$, then $m=2$ and $n\geq2$.
 \end{enumerate}
 \end{enumerate}
  $(2)$ Suppose $q=2$, so
 \begin{enumerate}[\rm\indent(i)]
   \item If $\left(\frac{p_1}{p_2}\right)=-1$, then $n=1$ and $m=2$.
   \item If $\left(\frac{p_1}{p_2}\right)=1$, then $m=2$ and $n\geq2$.
 \end{enumerate}
 \end{lem}

\section{Proofs of the main results}
 Recall first the following result from \cite[p. 205]{Gr-03}.
\begin{lem}\label{14}
 If $\mathcal{H}$ is an unramified ideal in some extension  $\KK/\kk=\kk(\sqrt{x})/\kk$, then the quadratic residue symbol is given by the Artin symbol $\varphi=\left(\frac{\kk(\sqrt{ x})/\kk}{\mathcal{H}}\right)$  as follows: $\left(\frac{x}{\mathcal{H}}\right)=\sqrt{x}^{\varphi-1}.$
 \end{lem}
\subsection{Proof of Theorem \ref{2}}
 (1) The assertion $\mathbf{C}l_2(\kk)=\langle[\mathcal{H}_0], [\mathcal{H}_1], [\mathcal{H}_2]\rangle\simeq(2, 2, 2)$ of  Theorem \ref{2} is proved in \cite{AT08} and \cite{AZT12-2}. In the following pages we will prove the other assertions.\\
(2) To prove the second assertion we will use the techniques that F. Lemmermeyer  has used in some of his works see for example \cite{Lm94} or \cite{Lm97}. Consider the following diagram
 \begin{figure}[H]
 $$
 \xymatrix@R=0.22cm@C=0.4cm{&&&& F_1=\QQ(\sqrt{p_1p_2}, i)\ar[ddrr]\\
 \\
&&\QQ(i) \ar[ddrr]\ar[uurr] \ar[rr] && \kk=\QQ(\sqrt{2p_1p_2}, i) \ar[rr] &&  \KK_3=\QQ(\sqrt{2}, \sqrt{p_1p_2}, i)\\
\\
&&&&F_2=\QQ(\sqrt{2}, i)\ar[uurr]  }
$$
\caption{Subfields of $\KK_3/\QQ(i)$}\label{10}
\end{figure}

Compute first $N_{\KK_3/\kk}(\mathbf{C}l_2(\KK_3))$. Recall that $$N_{\KK_3/\kk}(\mathbf{C}l_2(\KK_3))=\{[\mathcal{H}]
 \in\mathbf{C}l_2(\kk)/\displaystyle\left(\frac{2}{[\mathcal{H}]}\right)=1\}.$$
As  $\mathcal{H}_1$ and $\mathcal{H}_2$ are unramified in $\KK_3/\kk=\kk(\sqrt 2)/\kk=\kk(\sqrt{p_1p_2})/\kk$, so Lemma \ref{14} yields that
$\displaystyle\left(\frac{ 2}{\mathcal{H}_1\mathcal{H}_2}\right)
 =\displaystyle\left(\frac{ 2}{\mathcal{H}_1}\right)\displaystyle\left(\frac{ 2}{\mathcal{H}_2}\right)
  =\displaystyle\left(\frac{2}{p_1}\right)\displaystyle\left(\frac{2}{p_1}\right)
  =1.$
  On the other hand, $2$ ramifies completely  in $\kk/\QQ$ and splits in $F_1/\QQ$; moreover  $\mathcal{H}_0$ is unramified in $\KK_3/\kk$,  then $\mathcal{H}_0$ splits in $\KK_3$ i.e. $\mathcal{H}_0\in N_{\KK_3/\kk}(\mathbf{C}l_2(\KK_3))$. Thus
$$N_{\KK_3/\kk}(\mathbf{C}l_2(\KK_3))=\langle[\mathcal{H}_0], [\mathcal{H}_1\mathcal{H}_2]\rangle.$$

 \indent To this end, it is easy to see that $\KK_3/F_1$ and $\KK_3/F_2$ are ramified extensions, whereas  $\KK_3/\kk$  is not; so from  the class field theory $ [\mathbf{C}l_2(\kk):N_{\KK_3/\kk}(\mathbf{C}l_2(\KK_3))]=2$, $\mathbf{C}l_2(F_2)=N_{\KK_3/F_2}(\mathbf{C}l_2(\KK_3))$ and  $\mathbf{C}l_2(F_1)=N_{\KK_3/F_1}(\mathbf{C}l_2(\KK_3))$, hence Theorem \ref{9} implies that $$N_{\KK_3/F_1}(\mathbf{C}l_2(\KK_3))=\langle[\mathfrak{2}_{F_1}], [I]\rangle.$$
 Then there exists an ideal $\mathfrak{P}\in\KK_3$ such that $N_{\KK_3/F_1}(\mathfrak{P})\sim I$ and $N_{\KK_3/\kk}(\mathfrak{P})\in \langle[\mathcal{H}_0], [\mathcal{H}_1\mathcal{H}_2]\rangle$. One shows that $N_{\KK_3/\kk}(\mathfrak{P})\sim \mathcal{H}_1\mathcal{H}_2$ (see Lemma \ref{15} below). We claim that
     $$\left\{
 \begin{array}{ll}
 \mathfrak{P}^{2}\sim I, \text{ if } q=1,\\
\mathfrak{P}^{2}\sim \mathcal{H}_1\mathcal{H}_2I, \text{ if }  q=2.
 \end{array}\right.$$
 Let $t$ and $s$ be the elements of order $2$ of $\mathrm{Gal}(\KK_3/\QQ(i))$ which  fix $F_1$ and $\kk$, respectively. Using the identity    $2+(1+t+s+ts)=(1+t)+(1+s)+(1+ts)$  of the group ring $\ZZ[\mathrm{Gal}(\KK_3/\QQ(i))]$ and observing that the class numbers of $\QQ(i)$,  $F_2$ are odd, we find that
$$\mathfrak{P}^2\sim \mathfrak{P}^{1+t}\mathfrak{P}^{1+s}\mathfrak{P}^{1+ts}\sim \mathcal{H}_1\mathcal{H}_2I.$$
 As $\mathcal{H}_1\mathcal{H}_2\sim 1$, in $\mathbf{C}l_2(\KK_3)$, if $q=1$ (see Lemma \ref{7}), so the result claimed. Moreover $\mathfrak{p}_1\mathcal{O}_\kk=\mathcal{H}_1\mathcal{H}_2$, so  Lemma \ref{7} yields that in $\mathbf{C}l_2(\KK_3)$ we have
    $$\left\{
 \begin{array}{ll}
 \mathfrak{P}^{2^{m}}\sim I^{2^{m-1}}\sim\mathfrak{p}_1\sim1, \text{ if } q=1,\\
\mathfrak{P}^{2^{m}}\sim I^{2^{m-1}}\sim\mathfrak{p}_1, \text{ if }  q=2.
 \end{array}\right.$$
 On the other hand,  $\mathfrak{2}_{F_1}$  ramifies and $\mathcal{H}_0$ splits in  $\KK_3$, let $\mathfrak{A}$ be  an ideal of  $\KK_3$ above $\mathfrak{2}_{F_1}$, then $N_{\KK_3/\kk}(\mathfrak{A})\sim \mathcal{H}_0$ and  $N_{\KK_3/F_1}(\mathfrak{A})\sim \mathfrak{2}_{F_1}$. Thus, in $\mathbf{C}l_2(\KK_3)$, we have $$\mathfrak{A}^2\sim\mathfrak{2}_{F_1} \text{ and  }\mathfrak{A}^{2^{n+1}}\sim\mathfrak{2}_{F_1}^{2^{n}}.$$
 Recall that  $\mathcal{H}_j$ and  $\mathcal{P}_j$ coincide and remain inert in $\KK_3$; moreover  $\mathfrak{p}_1\mathcal{O}_{\KK_3}=\mathcal{H}_1\mathcal{H}_2\mathcal{O}_{\KK_3}=\mathcal{P}_1\mathcal{P}_2\mathcal{O}_{\KK_3}$  and $\mathfrak{p}_2\mathcal{O}_{\KK_3}=\mathcal{H}_3\mathcal{H}_4\mathcal{O}_{\KK_3}=\mathcal{P}_3\mathcal{P}_4\mathcal{O}_{\KK_3}$,
  therefore:\\
   $\bullet$ If $\left(\frac{p_1}{p_2}\right)=-1$ or $\left(\frac{p_1}{p_2}\right)=1$ and $N(\varepsilon_{p_1p_2})=-1$, then Lemmas  \ref{6}, \ref{7} and Theorem  \ref{9}  imply that
   $$\left\{
 \begin{array}{ll}
 \mathfrak{A}^{2^{n+1}}\sim\mathfrak{2}_{F_1}^{2^{n}}\sim1, \text{ if } q=1,\\
 \mathfrak{A}^{2^{n+1}}\sim\mathfrak{2}_{F_1}^{2^{n}}\not\sim1 \text{ and }\mathfrak{A}^{2^{n+2}}\sim\mathfrak{2}_{F_1}^{2^{n+1}}\sim1, \text{ if }  q=2.
 \end{array}\right.$$
$\bullet$ If $\left(\frac{p_1}{p_2}\right)=1$ and $N(\varepsilon_{p_1p_2})=1$, then $q=1$, and Lemma \ref{16} yields that $\mathcal{P}_1\mathcal{P}_3\sim \mathfrak{2}_{F_1}^{2^{n}}$.\\
 \indent Let us prove that $\mathcal{P}_1\mathcal{P}_3\mathcal{O}_{\KK_3}=\mathcal{H}_1\mathcal{H}_3\mathcal{O}_{\KK_3}$ is principal. We know that $N(\varepsilon_{2p_1p_2})=-1$, so the decomposition uniqueness in $\ZZ[i]$ implies that there exist $y_1$, $y_2$ in $\ZZ[i]$ such that
\begin{center}
  $\sqrt{\varepsilon_d}=
  \frac{1}{2}[y_1(1+i)\sqrt{(1\pm i)\pi_1\pi_3}+y_2(1-i)\sqrt{(1\mp i)\pi_2\pi_4}]$  (a) or\\
 $\sqrt{\varepsilon_d}=\frac{1}{2}[y_1(1+i)\sqrt{(1\pm i)\pi_1\pi_4}+y_2(1-i)\sqrt{(1\mp i)\pi_2\pi_3}]$ (b).
 \end{center}
   Moreover the ideal $\mathcal{H}_0\mathcal{H}_1\mathcal{H}_3$ is principal in $\kk$, if and only if there exists a unit $\varepsilon\in\kk$ such that   \begin{equation}\label{55}(1+i)\pi_1\pi_3\varepsilon=\alpha^2,\end{equation} where $\alpha\in\kk$. As $N(\varepsilon_{2p_1p_2})=-1$, so  Lemma
 \ref{6} involves that  $Q_\kk=1$, the unit index of $\kk$; hence $\varepsilon$ is  either real or purely imaginary.\par
Put $\alpha=\alpha_1+i\alpha_2$,  with $\alpha_1$,
 $\alpha_2\in\QQ(\sqrt{2p_1p_2})$, and suppose  $\varepsilon$ is  real (same proof if it is purely imaginary); as $\pi_1\pi_3=(e+2if)(g+2ih)=(eg-4fh)+2i(eh+gf)$, so the equation (\ref{55}) is equivalent to
$$\alpha_1^2-\alpha_2^2+2i\alpha_1\alpha_2=\varepsilon[(eg-4fh)-2(eh+fg)]+i\varepsilon_d[(eg-4fh)+2(eh+gf)],$$
hence
$$\left\{\begin{array}{ll}
   \alpha_1^2-\alpha_2^2 &=\varepsilon[(eg-4fh)-2(eh+fg)], \\
   2\alpha_1\alpha_2 & =\varepsilon[(eg-4fh)+2(eh+gf)],
 \end{array}
   \right.$$  so we get
 $\alpha_2=\frac{\varepsilon[(eg-4fh)+2(eh+gf)]}{2\alpha_1},$ thus
 $$4\alpha_1^4-4\varepsilon[(eg-4fh)-2(eh+fg)]\alpha_1^2-[(eg-4fh)+2(eh+fg)]^2\varepsilon^2=0,$$
the discriminant of this equation is $\Delta'=4\varepsilon^2d$, $d=2p_1p_2$,
 which implies that $$\alpha_1^2=\frac{\varepsilon}{4}[2[(eg-4fh)-2(eh+fg)]\pm2\sqrt d].$$
 Since \begin{center} $(1+i)\pi_1\pi_3+(1-i)\pi_2\pi_4 =2(eg-4fh)-4(eh+fg)$ and \\
          $\sqrt d  =\sqrt{(1-i)\pi_1\pi_3}\sqrt{(1+i)\pi_2\pi_4},$ \end{center}
   then
   $$\begin{array}{ll}
   \alpha_1^2&=\frac{\varepsilon}{4}(\sqrt{(1-i)\pi_1\pi_3}+\sqrt{(1+i)\pi_2\pi_4})^2,\text{ so }\\
    \alpha_1&=\frac{\sqrt{\varepsilon}}{2}(\sqrt{(1-i)\pi_1\pi_3}+\sqrt{(1+i)\pi_2\pi_4}),
    \end{array}$$
   therefore  if $\varepsilon=\varepsilon_d$ and $\sqrt{\varepsilon_d}$ takes the value  (a), we get
 $$
   \alpha_1 =\frac{1}{4}(2y_1\pi_1\pi_3+2y_2\pi_2\pi_4+(y_1(1+i)+y_2(1-i))\sqrt d),$$ and
   $$\alpha_2 =\frac{\varepsilon_d[(eg-4fh)+2(eh+gf)]}{2\alpha_1},
     $$
    and it is easy to see that $\alpha_1$, $\alpha_2$
   $\in\QQ(\sqrt{2p_1p_2})$; hence $\mathcal{H}_0\mathcal{H}_1\mathcal{H}_3$ is principal in $\kk$.  Proceeding  similarly,  we prove that
  $\mathcal{H}_0\mathcal{H}_2\mathcal{H}_3$ is principal in $\kk$ if $\sqrt{\varepsilon_d}$ takes the value  (b). Hence, in $\mathbf{C}l_2(\kk)$, we have $\mathcal{H}_3\sim\mathcal{H}_0\mathcal{H}_1$ or $\mathcal{H}_3\sim\mathcal{H}_0\mathcal{H}_2$, this in turn shows that, in $\mathbf{C}l_2(\KK_3)$, $\mathcal{H}_1\mathcal{H}_3\sim\mathcal{H}_0\mathcal{H}_1\mathcal{H}_2$ or $\mathcal{H}_1\mathcal{H}_3\sim\mathcal{H}_0$, as we know that $\mathcal{H}_0$, $\mathcal{H}_1\mathcal{H}_2$ capitulate in $\KK_3$, so the result. Thus
  $$\mathfrak{A}^{2^{n+1}}\sim\mathfrak{2}_{F_1}^{2^{n}}\sim1.$$
Consequently, in $\mathbf{C}l_2(\KK_3)$, we have
   $$\left\{
 \begin{array}{ll}
 \mathfrak{A}^{2^{n+1}}\sim\mathfrak{P}^{2^{m}}\sim \mathcal{H}_1\mathcal{H}_2 \sim 1, \text{ if } q=1,\\
 \mathfrak{A}^{2^{n+1}}\sim\mathfrak{P}^{2^{m}}\sim \mathcal{H}_1\mathcal{H}_2 \not\sim 1 \text{ and }\mathfrak{A}^{2^{n+2}}\sim\mathfrak{P}^{2^{m+1}}\sim1, \text{ if }  q=2.
 \end{array}\right.$$
To this end, note that for all  $i\leq n$, $j\leq m-1$, we have $\mathfrak{A}^{2^{i}}\mathfrak{P}^{2^{j}}\not\sim1$, as otherwise, we would have, by applying the norm $N_{\KK_3/F_1}$, $\mathfrak{2}_{F_1}^{2^{i}}I^{2^{j}}\sim1$, which contradicts  the results of Theorem \ref{9}.\\
{ \em Conclusion}\\
\indent If $q=1$, then $\langle[\mathfrak{A}], [\mathfrak{P}]\rangle$ is a subgroup of  $\mathbf{C}l_2(\KK_3)$ of type $(2^m, 2^{n+1})$, and as in this case $h(\KK_3)=h(p_1p_2)h(-p_1p_2)=2^{n+m+1}$ (see Lemma \ref{6}), so $$\mathbf{C}l_2(\KK_3)=\langle[\mathfrak{A}], [\mathfrak{P}]\rangle\simeq (2^{n+1}, 2^m).$$
\indent If $q=2$, then $\langle[\mathfrak{A}], [\mathfrak{P}]\rangle$ is a subgroup of $\mathbf{C}l_2(\KK_3)$ of type\\ $(2^{\min(m, n+1)}, 2^{\max(n+2, m+1)})$, and as in this case $h(\KK_3)=2h(p_1p_2)h(-p_1p_2)=2^{n+m+2}$ (see Lemma \ref{6}), so
 $$\mathbf{C}l_2(\KK_3)=\langle[\mathfrak{A}], [\mathfrak{P}]\rangle\simeq(2^{\min(m, n+1)}, 2^{\max(n+2, m+1)}).$$
\indent As $\mathcal{H}_1$,  $\mathcal{P}_1$ remain inert in $\KK_3$, so  they do not capitulate and coincide in $\KK_3$, hence they are of order $2$. Thus from Theorem  \ref{9} and Lemma \ref{16} we deduce that:\\
$\bullet$\    If $\left(\frac{p_1}{p_2}\right)=-1$ or $\left(\frac{p_1}{p_2}\right)=1$ and $N(\varepsilon_{p_1p_2})=-1$, then  $$\mathfrak{A}^{2^{n}}\mathfrak{P}^{2^{m-1}}\sim \mathfrak{2}_{F_1}^{2^{n-1}}I^{2^{m-2}}\sim\mathcal{P}_1\sim \mathcal{H}_1.$$
$\bullet$ If $\left(\frac{p_1}{p_2}\right)=1$ and $N(\varepsilon_{p_1p_2})=1$, then   $\mathfrak{2}_{F_1}^{2^{n}}\sim1$, thus
$$\mathfrak{P}^{2^{m-1}}\sim I^{2^{m-2}}\sim\mathcal{P}_1\sim \mathcal{H}_1.$$
 Finally,  from Lemma  \ref{12},  we deduce the following remark
\begin{rema}\label{13}
$(1)$ Assume $q=1$, so
\begin{enumerate}[\rm\indent(i)]
  \item If $\left(\frac{p_1}{p_2}\right)=-1$ or $\left(\frac{p_1}{p_2}\right)=1$ and $\left(\frac{p_1}{p_2}\right)_4\left(\frac{p_2}{p_1}\right)_4=-1$,   then $n=1$ and $m\geq3$, thus $\mathbf{C}l_2(\KK_3)\simeq(4,2^{m})$.
  \item If $\left(\frac{p_1}{p_2}\right)=1$ and $\left(\frac{p_1}{p_2}\right)_4\left(\frac{p_2}{p_1}\right)_4=1$,   then $m=2$ and $n\geq2$, hence  $\mathbf{C}l_2(\KK_3)\simeq(4,2^{n+1})$.
\end{enumerate}
$(2)$ If $q=2$, then  $\mathbf{C}l_2(\KK_3)\simeq(4,2^{n+2})$.
\end{rema}
This completes the proof of the second assertion.

(3) For the proof of the third assertion see \cite{AZT-3}.\par
(4) {\bf Computation of $\mathrm{Gal}(\kk_2^{(2)}/\kk)$}.\par
 Put $L=\kk_2^{(2)}$, the Hilbert 2-class field of $\kk$. Let $\displaystyle\left(\frac{L/\KK_3}{P}\right)$ denote the Artin symbol for the normal extension  $L/\KK_3$; hence it is clear that $\sigma=\displaystyle\left(\frac{L/\KK_3}{\mathfrak{P}}\right)$ and $\tau=\displaystyle\left(\frac{L/\KK_3}{\mathcal{\mathfrak{A}}}\right)$ generate the abelian subgroup  $\mathrm{Gal}(L/\KK_3)$ of $G=\mathrm{Gal}(L/\kk)$. If we put also  $\rho=\displaystyle\left(\frac{L/\kk}{\mathcal{H}_1}\right)$, then $\rho$ restricts to the nontrivial automorphism of  $\KK_3/\kk$, since $\mathcal{H}_1$ is not norm from $\KK_3/\kk$; from which we deduce that  $$G=\mathrm{Gal}(L/\kk)=\langle\rho, \tau, \sigma\rangle.$$
 Note  that $|G|=2|\mathrm{Gal}(L/\KK_3)|=
 \left\{
 \begin{array}{ll}
 2^{n+m+2} \text{ if } q=1,\\
 2^{n+m+3} \text{ if }  q=2.
 \end{array}\right.$\\
 To continue, let us prove the following result.
\begin{lem}\label{15}
 In $\mathbf{C}l_2(\kk)$, we have $N_{\KK_3/\kk}(\mathfrak{P})\sim \mathcal{H}_1\mathcal{H}_2$.
\end{lem}
\begin{proof}  We choose a prime ideal  $\mathfrak{R}$ in  $\KK_3$  such that $[\mathfrak{R}]=[\mathfrak{P}]$, this is always possible by  Chebotarev's theorem, hence  $\mathscr{R}_\kk \sim N_{\KK_3/\kk}(\mathfrak{R})$ and $\mathscr{R}_{F_1}\sim N_{\KK_3/F_1}(\mathfrak{R})$ are prime ideals in $\kk$ and $F_1$ respectively, thus $\mathscr{R}_{F_1}\sim N_{\KK_3/F_1}(\mathfrak{R})\sim N_{\KK_3/F_1}(\mathfrak{P})\sim I$. As the extension $F_1/k_1$ is ramified and $\mathbf{C}l_2(k_1)$ is generated by $\mathfrak{2}_1$,  we infer that the prime ideal  $\mathscr{R}_{k_1}\sim N_{F_1/k_1}(\mathscr{R}_{F_1})\sim N_{F_1/k_1}(I)\sim \mathfrak{2}_1^{2^i}$ with some integer $i$. This implies that $2^{2^i}r=\pm(x^2-p_1p_2y^2)$, which in turn shows that  $(\frac{r}{p_1})=1$.\\
\indent We know that   $N_{\KK_3/\kk}(\mathbf{C}l_2(\KK_3))=\langle[\mathcal{H}_0], [\mathcal{H}_1\mathcal{H}_2]\rangle$. So if $N_{\KK_3/\kk}(\mathfrak{P})\sim \mathcal{H}_0$, then $\mathscr{R}_\kk\sim\mathcal{H}_0$ (equivalence in $\mathbf{C}l_2(\kk)$); hence the prime ideal  $\mathfrak{r}=N_{\kk/k_0}(\mathscr{R}_\kk)$ of $k_0$ is equivalent, in $\mathbf{C}l_2(k_0)$, to $\widetilde{\mathfrak{2}}\sim P_1P_2$.
   Therefore the equivalence   $\mathfrak{r}\sim \widetilde{\mathfrak{2}}$ yields that  $ 2r=\pm(x^2-2p_1p_2y^2)$, where $x$, $y$ are in $\ZZ$; which shows that $(\frac{2r}{p_1})=1$, this leads to the contradiction $(\frac{r}{p_1})=-1$, since $(\frac{2}{p_1})=-1$. We get the same contradiction if we suppose that  $N_{\KK_3/\kk}(\mathfrak{P})\sim \mathcal{H}_0\mathcal{H}_1\mathcal{H}_2$. Finally the equivalence $N_{\KK_3/\kk}(\mathfrak{P})\sim 1$ can not occur since the order of $\sigma$ is strictly greater than 1.
\end{proof}
 Therefore the following relations hold:\\
$\bullet$  $[\tau, \sigma]=1$.\\
$\bullet$ $\rho^2=\displaystyle\left(\frac{L/\kk}{\mathcal{H}_1^2}\right)=
\displaystyle\left(\frac{L/\kk}{N_{\KK_3/\kk}(\mathcal{H}_1)}\right)=\displaystyle\left(\frac{L/\KK_3}{\mathcal{H}_1}\right)$, so $\rho^4=1$ .\\
$\bullet$ $\tau\rho^{-1}\tau\rho=\displaystyle\left(\frac{L/\KK_3}{\mathfrak{A}^{1+\rho}}\right)=1$, since  $\mathfrak{A}^{1+\rho}=N_{\KK_3/\kk}(\mathfrak{A})\sim \mathcal{H}_0\sim 1$, thus $[\tau, \rho]=\tau^{-1}\rho^{-1}\tau\rho=\tau^{-2}$ and $[\rho, \tau]=\tau^{2}$.\\
$\bullet$ $\sigma\rho^{-1}\sigma\rho=\displaystyle\left(\frac{L/\KK_3}{\mathfrak{P}^{1+\rho}}\right)=
\left\{
 \begin{array}{ll}
 1, \text{ if } q=1,\\
 \sigma^{2^m}, \text{ if }  q=2,
 \end{array}\right.
$ since, in $\mathbf{C}l_2(\KK_3)$, we have\\
  $\mathfrak{P}^{1+\rho}=N_{\KK_3/\kk}(\mathfrak{P})\sim \mathcal{H}_1\mathcal{H}_2\sim
  \left\{
 \begin{array}{ll}
 1, \text{ if } q=1,\\
\mathfrak{P}^{2^m}, \text{ if }  q=2,
 \end{array}\right.$\\
 therefore $[\sigma, \rho]=
   \left\{
 \begin{array}{ll}
 \sigma^{-2}, \text{ if } q=1,\\
\sigma^{2^m-2}=\sigma^{2}, \text{ if }  q=2, \text{ since in this case } m=2.
 \end{array}\right.$\\
$\bullet$ Suppose $q=1$, so\\
 - If $\left(\frac{p_1}{p_2}\right)=-1$ or $\left(\frac{p_1}{p_2}\right)=1$ and $N(\varepsilon_{p_1p_2})=-1$, then $\rho^4=\sigma^{2^m}=\tau^{2^{n+1}}=1$ and $\rho^2=\sigma^{2^{m-1}}\tau^{2^{n}}$, since $\mathfrak{A}^{2^{n+1}}\sim\mathfrak{P}^{2^{m}}\sim1$ and $\mathcal{H}_1\sim\mathfrak{A}^{2^{n}}\mathfrak{P}^{2^{m-1}}$.\\
 - If $\left(\frac{p_1}{p_2}\right)=1$ and $N(\varepsilon_{p_1p_2})=1$, then $\rho^4=\sigma^{2^m}=\tau^{2^{n+1}}=1$ and $\rho^2=\sigma^{2^{m-1}}$ since $\mathcal{H}_1\sim\mathfrak{P}^{2^{m-1}}$.\\
$\bullet$ Suppose $q=2$, so necessarily $\left(\frac{p_1}{p_2}\right)=-1$ or $\left(\frac{p_1}{p_2}\right)=1$ and $N(\varepsilon_{p_1p_2})=-1$, then\\
 $\left\{
   \begin{array}{ll}
   \rho^4=\sigma^{2^{m+1}}=\tau^{2^{n+2}}=1,\\
   \sigma^{2^{m}}=\tau^{2^{n+1}} \text{ and } \rho^2=\sigma^{2^{m-1}}\tau^{2^{n}},
   \end{array}\right.$\\
    since $\mathfrak{A}^{2^{n+2}}\sim\mathfrak{P}^{2^{m+1}}\sim1$, $\mathfrak{A}^{2^{n+1}}\sim\mathfrak{P}^{2^{m}}$ and $\mathcal{H}_1\sim\mathfrak{A}^{2^{n}}\mathfrak{P}^{2^{m-1}}$.\\
\indent (5) As $[\tau,\sigma]=1$, $[\rho,\sigma]=\sigma^{2}$ or $\sigma^{-2}$ and $[\rho, \tau]=\tau^{2}$, then the derived group of $G$ is $G'=\langle\sigma^2, \tau^2\rangle$, therefore\\
 $\mathbf{C}l_2(\kk^{(1)}_2)\simeq\left\{
   \begin{array}{ll}
   (2^{m-1}, 2^{n})  \text{ if } q=1,\\
   (2^{\min(m, n+1)-1}, 2^{\max(m+1, n+2)-1})=(2, 2^{n+1})  \text{ if } q=2.
   \end{array}\right.$\\
\indent (6) Finally, we compute the coclass of $G$.\\
Let $G$ be the group defined above. Then the lower central series of $G$ is defined inductively by $\gamma_1(G)=G$ and $\gamma_{i+1}(G)=[\gamma_i(G),G]$, that is the subgroup of $G$ generated by the set $\{[a, b]=a^{-1}b^{-1}ab/ a\in \gamma_i(G), b\in G\}$, so  the coclass of $G$ is defined to be $cc(G) = h-c$, where $|G|=2^h$ and  $c=c(G)$ is the nilpotency class of $G$, that is  the smallest  integer satisfying  $\gamma_{c+1}(G)=1$. We easily get \\
$\gamma_1(G)=G$.\\
$\gamma_2(G)=G'=\langle\sigma^2, \tau^2\rangle$.\\
$\gamma_3(G)=[G',G]=\langle\sigma^4, \tau^4\rangle$.\\
 Then   Proposition \ref{24}(6) (see below) yields that $\gamma_{j+1}(G)=[\gamma_j(G),G]=\langle\sigma^{2^{j}}, \tau^{2^{j}}\rangle$.\\
 \indent Suppose   $q=1$, then if we put $\upsilon=\max(n, m-1)$, we get\\ $\gamma_{\upsilon+2}(G)=\langle\sigma^{2^{\upsilon+1}}, \tau^{2^{\upsilon+1}}\rangle=\langle1\rangle$ and $\gamma_{\upsilon+1}(G)=\langle\sigma^{2^{\upsilon}}, \tau^{2^{\upsilon}}\rangle\neq\langle1\rangle$. As, in this case, $\mid G\mid=2^{n+m+2}$, so $$c(G)=\upsilon+1\text{ and } cc(G) = n+m+1-\upsilon=3,$$
 in fact, from Lemma \ref{12}, we have $m\geq3$ and $n=1$ or $m=2$ and $n\geq2$, so  the first case implies that $\upsilon=m-1$ and $cc(G) = n+m+1-\upsilon = 3$, whereas  the second one yields that $\upsilon=n$ and $cc(G) = n+m+1-\upsilon= 3.$\\
  \indent Suppose   $q=2$, then if we put $\upsilon=\max(n+1, m)$, we get\\ $\gamma_{\upsilon+2}(G)=\langle\sigma^{2^{\upsilon+1}}, \tau^{2^{\upsilon+1}}\rangle=\langle1\rangle$ and $\gamma_{\upsilon+1}(G)=\langle\sigma^{2^{\upsilon}}, \tau^{2^{\upsilon}}\rangle\neq\langle1\rangle$. As, in this case, $\mid G\mid=2^{n+m+3}$, so $$c(G)=\upsilon+1\text{ and } cc(G) = n+m+2-\upsilon=3,$$
 since in this case, from Lemma \ref{12},  $m=2$ and $n\geq1$, thus  $\upsilon=n+1$.
 \subsection{Proof of Theorem \ref{3}}
 For this  we need the following results which are easy to check.
\begin{propo}\label{24}
Let $G=\langle\sigma, \tau, \rho\rangle$ denote the group defined above, then
\begin{enumerate}[\rm\indent(1)]
  \item $\rho^{-1}\sigma\rho=\left\{\begin{array}{ll}
                              \sigma^{-1} \text{ if } q=1,\\
                              \sigma^3 \text{ if } q=2.
                              \end{array}
                              \right.$
  \item $\rho^{-1}\tau\rho=\tau^{-1}$.
  \item $[\rho^2,\sigma]=[\rho^2,\tau]=1$.
  \item $(\tau\rho)^2=\rho^2$.
  \item $(\sigma\tau\rho)^2=(\sigma\rho)^2=\left\{\begin{array}{ll}
                              \rho^{2} \text{ if } q=1,\\
                              \rho^2\sigma^{4} \text{ if } q=2.
                              \end{array}
                              \right.$
  \item For all $r\in\NN$, $[\rho, \tau^{2^r}]=\tau^{2^{r+1}}$ and
                              $[\rho,\sigma^{2^r}]=
                              \left\{\begin{array}{ll}
                              \sigma^{2^{r+1}} \text{ if } q=1,\\
                              \sigma^{-2^{r+1}} \text{ if } q=2.
                              \end{array}
                              \right.$
\end{enumerate}
\end{propo}
 The proof of  Theorem \ref{3} consists of 3 parts. In the first part, we will compute $N_{\KK_j/\kk}(\mathbf{C}l_2(\KK_j))$, for all $1\leq j\leq7$.  In the second one,  we will determine the capitulation kernels $\kappa_{\KK_j}$ and the types of $\mathbf{C}l_2(\KK_4)$ and in the third one, we will determine the capitulation kernels $\kappa_{\LL_j}$ and the types of $\mathbf{C}l_2(\LL_4)$. It should be noted that if $(\frac{p_1}{p_2})=-1$, then  Propositions \ref{5}, \ref{11} imply that
 $$\left\{\begin{array}{ll}
   q=1 \Leftrightarrow \left(\frac{\pi_1}{\pi_3}\right)=\left(\frac{1+i}{\pi_1}\right)\left(\frac{1+i}{\pi_3}\right) \\
   q=2  \Leftrightarrow \left(\frac{\pi_1}{\pi_3}\right)=-\left(\frac{1+i}{\pi_1}\right)\left(\frac{1+i}{\pi_3}\right)
   \end{array}
 \right.$$
 \subsubsection{ \bf Norm class groups} Let us  compute $N_j=N_{\KK_j/\kk}(\mathbf{C}l_2(\KK_j))$, the results are summarized in the following table.  Note that the left hand sides refer to the case  $\left(\frac{\pi_1}{\pi_3}\right)=-1$, while the right ones refer to the case  $\left(\frac{\pi_1}{\pi_3}\right)=1$. Put $B=\left(\frac{1+i}{\pi_1}\right)\left(\frac{1+i}{\pi_3}\right)$.
\footnotesize
 \begin{longtable}{| c | c | c |}
 \caption{Norm class groups}\label{21}\\
\hline
 $\KK_j$  & $N_j$ for  $(\frac{p_1}{p_2})=1$ & $N_j$ for  $(\frac{p_1}{p_2})=-1$\\
\hline
\endfirsthead
\hline
 $\KK_j$& $N_j$ for  $(\frac{p_1}{p_2})=1$ & $N_j$ for  $(\frac{p_1}{p_2})=-1$\\
\hline
\endhead
 $\KK_1$ & $\langle[\mathcal{H}_0\mathcal{H}_1], [\mathcal{H}_0\mathcal{H}_2]\rangle$ &  $\langle[\mathcal{H}_1], [\mathcal{H}_2]\rangle$\\ \hline
 $\KK_2$ &  $\langle[\mathcal{H}_1], [\mathcal{H}_2]\rangle$ & $\langle[\mathcal{H}_0\mathcal{H}_1], [\mathcal{H}_0\mathcal{H}_2]\rangle$\\ \hline
$\KK_3$ &  $\langle[\mathcal{H}_0], [\mathcal{H}_1\mathcal{H}_2]\rangle$ &  $\langle[\mathcal{H}_0], [\mathcal{H}_1\mathcal{H}_2]\rangle$\\ \hline
  &  $\langle[\mathcal{H}_0], [\mathcal{H}_2]\rangle$, $\langle[\mathcal{H}_0], [\mathcal{H}_1]\rangle$ if $B=1$ &  $\langle[\mathcal{H}_1], [\mathcal{H}_0\mathcal{H}_2]\rangle$ $\langle[\mathcal{H}_0], [\mathcal{H}_2]\rangle$ if $q=1$\\[-1ex]
 \raisebox{2ex}{$\KK_4$}
         &  $\langle[\mathcal{H}_2], [\mathcal{H}_0\mathcal{H}_1]\rangle$, $\langle[\mathcal{H}_1], [\mathcal{H}_0\mathcal{H}_2]\rangle$ if  $B=-1$ &  $\langle[\mathcal{H}_0], [\mathcal{H}_1]\rangle$ $\langle[\mathcal{H}_2], [\mathcal{H}_0\mathcal{H}_1]\rangle$ if $q=2$\\[1ex]\hline
  &  $\langle[\mathcal{H}_2], [\mathcal{H}_0\mathcal{H}_1]\rangle$, $\langle[\mathcal{H}_1], [\mathcal{H}_0\mathcal{H}_2]\rangle$ if $B=1$ &  $\langle[\mathcal{H}_0], [\mathcal{H}_2]\rangle$ $\langle[\mathcal{H}_1], [\mathcal{H}_0\mathcal{H}_2]\rangle$ if $q=1$\\[-1ex]
\raisebox{1.7ex}{$\KK_5$}
  &  $\langle[\mathcal{H}_0], [\mathcal{H}_2]\rangle$, $\langle[\mathcal{H}_0], [\mathcal{H}_1]\rangle$ if $B=-1$ &  $\langle[\mathcal{H}_2], [\mathcal{H}_0\mathcal{H}_1]\rangle$ $\langle[\mathcal{H}_0], [\mathcal{H}_1]\rangle$ if $q=2$\\[1ex]\hline
  &  $\langle[\mathcal{H}_1], [\mathcal{H}_0\mathcal{H}_2]\rangle$, $\langle[\mathcal{H}_2], [\mathcal{H}_0\mathcal{H}_1]\rangle$ if $B=1$ &  $\langle[\mathcal{H}_0], [\mathcal{H}_1]\rangle$ $\langle[\mathcal{H}_2], [\mathcal{H}_0\mathcal{H}_1]\rangle$ if $q=1$\\[-1ex]
\raisebox{2ex}{$\KK_6$}
         &  $\langle[\mathcal{H}_0], [\mathcal{H}_1]\rangle$, $\langle[\mathcal{H}_0], [\mathcal{H}_2]\rangle$ if $B=-1$ &  $\langle[\mathcal{H}_1], [\mathcal{H}_0\mathcal{H}_2]\rangle$ $\langle[\mathcal{H}_0], [\mathcal{H}_2]\rangle$ if $q=2$\\[1ex]\hline
  &  $\langle[\mathcal{H}_0], [\mathcal{H}_1]\rangle$, $\langle[\mathcal{H}_0], [\mathcal{H}_2]\rangle$ if $B=1$ &  $\langle[\mathcal{H}_2], [\mathcal{H}_0\mathcal{H}_1]\rangle$ $\langle[\mathcal{H}_0], [\mathcal{H}_1]\rangle$ if $q=1$\\[-1ex]
\raisebox{2ex}{$\KK_7$}
         &  $\langle[\mathcal{H}_1], [\mathcal{H}_0\mathcal{H}_2]\rangle$, $\langle[\mathcal{H}_2], [\mathcal{H}_0\mathcal{H}_1]\rangle$  if $B=-1$ &  $\langle[\mathcal{H}_0], [\mathcal{H}_2]\rangle$ $\langle[\mathcal{H}_1], [\mathcal{H}_0\mathcal{H}_2]\rangle$ if $q=2$\\[1ex]\hline
\end{longtable}
\normalsize
To check the  table entries we use  Lemmas \ref{4}, \ref{14}, Propositions \ref{5}, \ref{11} and the following results which are easy to prove.
\begin{lem}
Let $p_1\equiv p_2\equiv1\pmod4$ be primes. Put $p_1=\pi_1\pi_2$ and  $p_2=\pi_3\pi_4$, where $\pi_j\in\ZZ[i]$,  then
\begin{enumerate}[\rm\indent(i)]
  \item $\left(\frac{\pi_1}{\pi_2}\right)=\left(\frac{\pi_3}{\pi_4}\right)=
  \left\{\begin{array}{ll}
   1 \text{ if } p_1\equiv p_2\equiv1\pmod8,\\
   -1  \text{ if } p_1\equiv p_2\equiv5\pmod8
   \end{array}
 \right.$
  \item If  $\left(\frac{p_1}{p_2}\right)=1$, then $\left(\frac{\pi_1}{\pi_3}\right)=\left(\frac{\pi_2}{\pi_3}\right)=
\left(\frac{\pi_1}{\pi_4}\right)=\left(\frac{\pi_2}{\pi_4}\right)$.
  \item If  $\left(\frac{p_1}{p_2}\right)=-1$, then $\left(\frac{\pi_1}{\pi_3}\right)=\left(\frac{\pi_2}{\pi_4}\right)=-\left(\frac{\pi_2}{\pi_3}\right)=-\left(\frac{\pi_1}{\pi_4}\right)$.
  \item If  $\left(\frac{2}{p_1}\right)=1$, then $\left(\frac{1+i}{\pi_1}\right)=\left(\frac{1+i}{\pi_2}\right)$.
  \item If  $\left(\frac{2}{p_1}\right)=-1$, then $\left(\frac{1+i}{\pi_1}\right)=-\left(\frac{1+i}{\pi_2}\right)$.
\end{enumerate}
\end{lem}
Compute $N_j$ in a few cases. Keeping in mind that  $\mathcal{H}_0$, $\mathcal{H}_1$ and  $\mathcal{H}_2$ are unramified in $\KK_j/\kk$.\\
 Take first $\KK_1=\kk(\sqrt{p_1})=\kk(\sqrt{2p_2})= \QQ(\sqrt p_1, \sqrt{2p_2}, i)$. As  $N_1=
  \{[\mathcal{H}]\in\mathbf{C}l_2(\kk)/\displaystyle\left(\frac{\alpha}{\mathcal{H}}\right)=1\}$, so for $j\in\{1, 2\}$ we get
 \begin{align*}
 \left(\frac{\kk(\sqrt{ 2p_2})/\kk}{\mathcal{H}_j}\right)&=\left(\frac{\kk(\sqrt{ 2p_2})/\kk}{\mathcal{H}_j}\right)(\sqrt{ 2p_2})(\sqrt{ 2p_2})^{-1}\\
 &=\left(\frac{2p_2}{\mathcal{H}_j}\right)\\
  &=\left(\frac{2p_2}{p_1}\right)\\
  &=\left(\frac{2}{p_1}\right)\left(\frac{p_1}{p_2}\right).
\end{align*}

Thus

  $\bullet$ If $\left(\frac{p_1}{p_2}\right)=-1$, then $[\mathcal{H}_j]\in N_1$, hence
  $N_1=\langle[\mathcal{H}_1],  \mathcal{H}_2]\rangle.$

 $\bullet$ If $\left(\frac{p_1}{p_2}\right)=1$, then $[\mathcal{H}_j]\not\in N_1$, hence
  $[\mathcal{H}_1\mathcal{H}_2]\in N_1$.  Moreover, since $\left(\frac{2}{p_1}\right)=-1$, then $\mathcal{H}_0\not\in N_1$; from which we deduce that  $[\mathcal{H}_0\mathcal{H}_1]$ and $[\mathcal{H}_0\mathcal{H}_2]$ are in  $N_1$, therefore
  $$N_1=\langle[\mathcal{H}_0\mathcal{H}_1],  [\mathcal{H}_0\mathcal{H}_2]\rangle.$$
\indent  Take an other example,   $\KK_4=\kk(\sqrt{\pi_1\pi_3})=\kk(\sqrt{2\pi_2\pi_4})$. First prove that   $\left(\frac{\pi_1\pi_3}{\mathcal{H}_0}\right)=\left(\frac{1+i}{\pi_1}\right)\left(\frac{1+i}{\pi_3}\right)$. As  $1+i$ is unramified in both of $\QQ(\sqrt{\pi_1\pi_3})/\QQ(i)$ and  $\kk/\QQ(i)$, so   according to \cite[Proposition 4.2, p.112]{Lm00} and Hilbert symbol properties we get
 \begin{align*}
 \left(\frac{\pi_1\pi_3}{\mathcal{H}_0}\right)=\left(\frac{ \pi_1\pi_3}{1+i}\right)
 &=\left(\frac{ \pi_1\pi_3}{1+i}\right)^{v_{1+i}(1+i)}\\
  &=\left(\frac{1+i, \pi_1\pi_3}{1+i}\right)\\
  &=\left(\frac{1+i, \pi_1}{1+i}\right)\left(\frac{1+i, \pi_3}{1+i}\right).
\end{align*}
On the other hand, the product formula implies, for $j\in\{1, 2\}$, that $$\left(\frac{1+i, \pi_j}{1+i}\right)\left(\frac{1+i, \pi_j}{\pi_j}\right)\prod_{\mathcal{P}\neq \pi_j,
                        \mathcal{P}\neq 1+i}\left(\frac{1+i, \pi_j}{\mathcal{P}}\right)=1;$$
 as $\mathcal{P}$ does not divide $\pi_j$ and $1+i$, so  $\left(\frac{1+i,\ \pi_j}{\mathcal{P}}\right)=1$, which yields that

 $$\left(\frac{1+i, \pi_j}{1+i}\right)\left(\frac{1+i, \pi_j}{\pi_j}\right)=1,\text{ hence }
 \left(\frac{1+i, \pi_j}{1+i}\right)=\left(\frac{1+i, \pi_j}{\pi_j}\right)=\left(\frac{1+i}{\pi_j}\right).$$
 This implies the result.

 Compute now   $N_4$.\\ We have
 $\left\{
 \begin{array}{ll}
 \left(\frac{\pi_1\pi_3}{\mathcal{H}_2}\right)=\left(\frac{\pi_1\pi_3}{\pi_2}\right)=\left(\frac{\pi_1}{\pi_2}\right)\left(\frac{\pi_3}{\pi_2}\right)=-\left(\frac{\pi_2}{\pi_3}\right),\\
 \left(\frac{2\pi_2\pi_4}{\mathcal{H}_1}\right)=\left(\frac{2\pi_2\pi_4}{\pi_1}\right)=\left(\frac{2}{p_1}\right)\left(\frac{\pi_1}{\pi_2}\right)\left(\frac{\pi_4}{\pi_1}\right)=\left(\frac{\pi_1}{\pi_4}\right),\\
 \left(\frac{\pi_1\pi_3}{\mathcal{H}_0}\right)=\left(\frac{1+i}{\pi_1}\right)\left(\frac{1+i}{\pi_3}\right).
 \end{array}\right.$\\
 Assume that  $\left(\frac{p_1}{p_2}\right)=1$. So \\
$\bullet$ If  $\left(\frac{p_1}{p_2}\right)_4\left(\frac{p_2}{p_1}\right)_4=\left(\frac{\pi_1}{\pi_3}\right)=-1$, then $\mathcal{H}_2\in N_4$ and $\mathcal{H}_1\not\in N_4$;\\
thus
  $\left\{
 \begin{array}{ll}
\text{ If } \left(\frac{1+i}{\pi_1}\right)\left(\frac{1+i}{\pi_3}\right)=1, \text{ then } N_4=\langle[\mathcal{H}_0], [\mathcal{H}_2]\rangle,\\
\text{ If } \left(\frac{1+i}{\pi_1}\right)\left(\frac{1+i}{\pi_3}\right)=-1, \text{ then } N_4=\langle[\mathcal{H}_2], [\mathcal{H}_0\mathcal{H}_1]\rangle.
 \end{array}\right.$\\
 $\bullet$ If  $\left(\frac{p_1}{p_2}\right)_4\left(\frac{p_2}{p_1}\right)_4=\left(\frac{\pi_1}{\pi_3}\right)=1$, then $\mathcal{H}_2\not\in N_4$ and $\mathcal{H}_1\in N_4$;\\
 hence
  $\left\{
 \begin{array}{ll}
\text{ If  } \left(\frac{1+i}{\pi_1}\right)\left(\frac{1+i}{\pi_3}\right)=1, \text{ then } N_4=\langle[\mathcal{H}_0], [\mathcal{H}_1]\rangle,\\
\text{ If } \left(\frac{1+i}{\pi_1}\right)\left(\frac{1+i}{\pi_3}\right)=-1, \text{ then } N_4=\langle[\mathcal{H}_1], [\mathcal{H}_0\mathcal{H}_2]\rangle.
 \end{array}\right.$\\
Assume that  $\left(\frac{p_1}{p_2}\right)=-1$. So\\
$\bullet$ If  $q=1$, then $\left(\frac{\pi_1}{\pi_3}\right)=\left(\frac{1+i}{\pi_1}\right)\left(\frac{1+i}{\pi_3}\right)$, hence\\
  $\left\{
 \begin{array}{ll}
\text{ If } \left(\frac{\pi_1}{\pi_3}\right)=1, \text{ then } N_4=\langle[\mathcal{H}_0], [\mathcal{H}_2]\rangle,\\
\text{ If } \left(\frac{\pi_1}{\pi_3}\right)=-1, \text{ then } N_4=\langle[\mathcal{H}_1], [\mathcal{H}_0\mathcal{H}_2]\rangle.
 \end{array}\right.$\\
$\bullet$ If  $q=2$, then $\left(\frac{\pi_1}{\pi_3}\right)=-\left(\frac{1+i}{\pi_1}\right)\left(\frac{1+i}{\pi_3}\right)$, hence\\
  $\left\{
 \begin{array}{ll}
\text{ If } \left(\frac{\pi_1}{\pi_3}\right)=1, \text{ then } N_4=\langle[\mathcal{H}_2], [\mathcal{H}_0\mathcal{H}_1]\rangle,\\
\text{ If } \left(\frac{\pi_1}{\pi_3}\right)=-1, \text{ then } N_4=\langle[\mathcal{H}_0], [\mathcal{H}_1]\rangle.
 \end{array}\right.$\\
Proceeding similarly,  we check the other table inputs.
\subsubsection{\bf Capitulation kernels $\kappa_{\KK_j}$ and $\mathrm{Gal}(\L/\KK_j)$\label{17}}
  Let us compute the  Galois groups $G_j=\mathrm{Gal}(\L/\KK_j)$,  the capitulation kernels $\kappa_{\KK_j}$,  $\kappa_{\KK_j}\cap N_j $ and the types of $\mathbf{C}l_2(\KK_j)$.  The results are summarized in the following tables.  Note that the left hand sides refer to the case  $\left(\frac{\pi_1}{\pi_3}\right)=-1$, while the right ones refer to the case  $\left(\frac{\pi_1}{\pi_3}\right)=1$. Put $B=\left(\frac{1+i}{\pi_1}\right)\left(\frac{1+i}{\pi_3}\right)$, $a=\min(m, n+1)$ and $b=\max(m+1, n+2)$.
\scriptsize
 \begin{longtable}{|c c | c | c | c | c |}
\caption{\small{ $\kappa_{\KK_j}$ for the case $\left(\frac{p_1}{p_2}\right)=1$}.\label{22}}\\
\hline
 $\KK_j$  && $G_j$ &  $\kappa_{\KK_j}$ & $\kappa_{\KK_j}\cap N_j $ & $\mathbf{C}l_2(\KK_j)$\\
\hline
\endfirsthead
\hline
 $\KK_j$ & &$G_j$ &  $\kappa_{\KK_j}$ & $\kappa_{\KK_j}\cap N_j$ & $\mathbf{C}l_2(\KK_j)$ \\
\hline
\endhead
 $\KK_1$ && $\langle\sigma, \tau\rho, \tau^2\rangle$  & $\langle[\mathcal{H}_1], [\mathcal{H}_2]\rangle$ & $\langle[\mathcal{H}_1\mathcal{H}_2]\rangle$ & $(2, 2, 2)$\\ \hline
  $\KK_2$ & & $\langle\sigma, \rho, \tau^2\rangle$  &   $\langle[\mathcal{H}_0\mathcal{H}_1], [\mathcal{H}_0\mathcal{H}_2]\rangle$ & $\langle[\mathcal{H}_1\mathcal{H}_2]\rangle$  & $(2, 2, 2)$\\ \hline
 & $q=1$ & &  $\langle[\mathcal{H}_0], [\mathcal{H}_1\mathcal{H}_2]\rangle$ & $N_3$  & $(2^{m}, 2^{n+1})$ \\[-1ex]
  \raisebox{2ex}{$\KK_3$} &$q=2$ & \raisebox{2ex}{$\langle\tau, \sigma\rangle$}
  &  $\langle[\mathcal{H}_0]\rangle$ &$\langle[\mathcal{H}_0]\rangle$  & $(2^{a}, 2^{b})$ \\ [1ex]\hline
& $B=1$ &  $\langle\tau, \rho\sigma, \sigma^2\rangle$ $\langle\tau, \rho\rangle$  &  $\langle[\mathcal{H}_0], [\mathcal{H}_1]\rangle$ &\hspace{-0.5cm}$\langle[\mathcal{H}_0]\rangle$  &\\[-1ex]
 \raisebox{2ex}{$\KK_4$}
  &  $B=-1$   &  $\langle\sigma\tau, \tau\rho, \sigma^2\rangle$ $\langle\sigma\tau, \rho\rangle$ &  $\langle[\mathcal{H}_1], [\mathcal{H}_0\mathcal{H}_2]\rangle$& $\langle[\mathcal{H}_0\mathcal{H}_1\mathcal{H}_2]\rangle$\  \raisebox{2ex}{$N_4$}& \raisebox{2ex}{ $(2, 2, 2)$ $(2, 4)$}\\[1ex]\hline
\raisebox{-2ex}{$\KK_5$} & $B=1$ &  $\langle\sigma\tau, \tau\rho, \sigma^2\rangle$ $\langle\sigma\tau, \rho\rangle$  & $\langle[\mathcal{H}_1], [\mathcal{H}_0\mathcal{H}_2]\rangle$ & $\langle[\mathcal{H}_0\mathcal{H}_1\mathcal{H}_2]\rangle$ \raisebox{-2ex}{$N_5$} & \raisebox{-2ex}{ $(2, 2, 2)$ $(2, 4)$}\\[-1ex]
 & $B=-1$  & $\langle\tau, \rho\sigma, \sigma^2\rangle$ $\langle\tau, \rho\rangle$  & $\langle[\mathcal{H}_0], [\mathcal{H}_1]\rangle$ & \hspace{-0.5cm}$\langle[\mathcal{H}_0]\rangle$  &  \\[1ex]\hline
 \raisebox{-2ex}{$\KK_6$}&$B=1$ &  $\langle\sigma\tau, \rho, \sigma^2\rangle$ $\langle\sigma\tau, \tau\rho\rangle$  & $\langle[\mathcal{H}_2], [\mathcal{H}_0\mathcal{H}_1]\rangle$ & $\langle[\mathcal{H}_0\mathcal{H}_1\mathcal{H}_2]\rangle$ \raisebox{-2ex}{$N_6$} &\raisebox{-2ex}{ $(2, 2, 2)$ $(2, 4)$}\\[-1ex]
  &   $B=-1$    &  $\langle\tau, \rho, \sigma^2\rangle$  $\langle\tau, \rho\sigma\rangle$  & $\langle[\mathcal{H}_0], [\mathcal{H}_2]\rangle$ & \hspace{-0.5cm}$\langle[\mathcal{H}_0]\rangle$ & \\[1ex]\hline
& $B=1$ &  $\langle\tau, \rho, \sigma^2\rangle$  $\langle\tau, \rho\sigma\rangle$ & $\langle[\mathcal{H}_0], [\mathcal{H}_2]\rangle$& \hspace{-0.5cm}$\langle[\mathcal{H}_0]\rangle$ &\\[-1ex]
\raisebox{2ex}{$\KK_7$ }
    & $B=-1$    &  $\langle\sigma\tau, \rho, \sigma^2\rangle$  $\langle\sigma\tau, \tau\rho\rangle$  & $\langle[\mathcal{H}_2], [\mathcal{H}_0\mathcal{H}_1]\rangle$ & $\langle[\mathcal{H}_0\mathcal{H}_1\mathcal{H}_2]\rangle$ \raisebox{2ex}{$N_7$}& \raisebox{2ex}{ $(2, 2, 2)$ $(2, 4)$} \\[1ex]\hline
\end{longtable}
 \begin{longtable}{| c c | c | c | c | c |}
\caption{\small{ $\kappa_{\KK_j}$ for the case $\left(\frac{p_1}{p_2}\right)=-1$.\label{23}}}\\
\hline
 $\KK_j$ & & $G_j$ &  $\kappa_{\KK_j}$ & $\kappa_{\KK_j}\cap N_j$ & $\mathbf{C}l_2(\KK_j)$\\
\hline
\endfirsthead
\hline
 $\KK_j$ & & $G_j$ &  $\kappa_{\KK_j}$ & $\kappa_{\KK_j}\cap N_j$ & $\mathbf{C}l_2(\KK_j)$ \\
\hline
\endhead
 $\KK_1$ & & $\langle\sigma, \rho\rangle$  & $\langle[\mathcal{H}_1], [\mathcal{H}_2]\rangle$ & $N_1$ & $(2, 4)$\\ \hline
  $\KK_2$ & & $\langle\sigma, \tau\rho\rangle$  &   $\langle[\mathcal{H}_0\mathcal{H}_1], [\mathcal{H}_0\mathcal{H}_2]\rangle$ & $N_2$  & $(2, 4)$\\ \hline
  & $q=1$ & & $\langle[\mathcal{H}_0], [\mathcal{H}_1\mathcal{H}_2]\rangle$ & $N_3$ & $(4, 2^{m})$ \\[-1ex]
  \raisebox{2ex}{$\KK_3$} &$q=2$ & \raisebox{2ex}{$\langle\tau, \sigma\rangle$}
  &  $\langle[\mathcal{H}_0]\rangle$ &  $\langle[\mathcal{H}_0]\rangle$ & $(4, 2^{m+1})$ \\ [1ex]\hline
& $q=1$ &  $\langle\rho, \tau\sigma\rangle$ $\langle\tau, \rho\sigma, \sigma^2\rangle$  &  $\langle[\mathcal{H}_1], [\mathcal{H}_0\mathcal{H}_2]\rangle$  $\langle[\mathcal{H}_0], [\mathcal{H}_1]\rangle$ & \qquad $\langle[\mathcal{H}_0]\rangle$ &\\[-1ex]
 \raisebox{2ex}{$\KK_4$}
 & $q=2$     &  $\langle\tau, \rho\rangle$ $\langle\sigma\tau, \tau\rho, \sigma^2\rangle$  & $\langle[\mathcal{H}_0], [\mathcal{H}_1]\rangle$ $\langle[\mathcal{H}_1], [\mathcal{H}_0\mathcal{H}_2]\rangle$ & \raisebox{2ex}{$N_4$} $\langle[\mathcal{H}_0\mathcal{H}_1\mathcal{H}_2]\rangle$  & \raisebox{2ex}{ $(2, 4)$ $(2, 2, 2)$}\\[1ex]\hline
 & $q=1$ &  $\langle\tau, \rho\sigma, \sigma^2\rangle$ $\langle\rho, \sigma\tau\rangle$ & $\langle[\mathcal{H}_0], [\mathcal{H}_1]\rangle$ $\langle[\mathcal{H}_1], [\mathcal{H}_0\mathcal{H}_2]\rangle$ & \hspace{-0.5cm}$\langle[\mathcal{H}_0]\rangle$  &\\[-1ex]
\raisebox{2ex}{$\KK_5$}
& $q=2$   & $\langle\sigma\tau, \tau\rho, \sigma^2\rangle$ $\langle\tau, \rho\rangle$  &  $\langle[\mathcal{H}_1], [\mathcal{H}_0\mathcal{H}_2]\rangle$  $\langle[\mathcal{H}_0], [\mathcal{H}_1]\rangle$ & $\langle[\mathcal{H}_0\mathcal{H}_1\mathcal{H}_2]\rangle$ \raisebox{2ex}{$N_5$} & \raisebox{2ex}{ $(2, 2, 2)$ $(2, 4)$}\\[1ex]\hline
& $q=1$ &  $\langle\tau, \rho, \sigma^2\rangle$ $\langle\sigma\tau, \tau\rho\rangle$  & $\langle[\mathcal{H}_0], [\mathcal{H}_2]\rangle$ $\langle[\mathcal{H}_2], [\mathcal{H}_0\mathcal{H}_1]\rangle$ & \hspace{-0.5cm}$\langle[\mathcal{H}_0]\rangle$ &\\[-1ex]
\raisebox{2ex}{$\KK_6$}
  & $q=2$       &  $\langle\rho, \sigma\tau, \sigma^2\rangle$  $\langle\tau, \rho\sigma\rangle$  & $\langle[\mathcal{H}_2], [\mathcal{H}_0\mathcal{H}_1]\rangle$ $\langle[\mathcal{H}_0], [\mathcal{H}_2]\rangle$ & $\langle[\mathcal{H}_0\mathcal{H}_1\mathcal{H}_2]\rangle$ \raisebox{2ex}{$N_6$}  & \raisebox{2ex}{ $(2, 2, 2)$ $(2, 4)$}\\[1ex]\hline
 & $q=1$ &  $\langle\sigma\tau, \tau\rho\rangle$  $\langle\rho, \tau, \sigma^2\rangle$  &$\langle[\mathcal{H}_2], [\mathcal{H}_0\mathcal{H}_1]\rangle$  $\langle[\mathcal{H}_0], [\mathcal{H}_2]\rangle$ & \hspace{0.5cm}$\langle[\mathcal{H}_0]\rangle$&\\[-1ex]
\raisebox{2ex}{$\KK_7$ }
  &   $q=2$    &  $\langle\tau, \rho\sigma\rangle$  $\langle\rho, \sigma\tau, \sigma^2\rangle$  & $\langle[\mathcal{H}_0], [\mathcal{H}_2]\rangle$  $\langle[\mathcal{H}_2], [\mathcal{H}_0\mathcal{H}_1]\rangle$ & \raisebox{2ex}{$N_7$ } $\langle[\mathcal{H}_0\mathcal{H}_1\mathcal{H}_2]\rangle$  & \raisebox{2ex}{ $(2, 4)$ $(2, 2, 2)$} \\[1ex]\hline
\end{longtable}
\normalsize
Before proving these results, note that,  from   Tables \ref{21}, \ref{22} and \ref{23} we get the following remark:
\begin{rema}
Put $B=(\frac{1+i}{\pi_1})(\frac{1+i}{\pi_3})$ and $\pi=(\frac{\pi_1}{\pi_3})$.
  \begin{enumerate}[\rm\indent(1)]
\item
    $\left\{\begin{array}{ll}
\kappa_{ \KK_1}=N_2,\  \
   \kappa_{ \KK_2}=N_1 \text{ if } (\frac{p_1}{p_2})=1,\\
\kappa_{ \KK_1}=N_1,\ \
   \kappa_{ \KK_2}=N_2 \text{ if } (\frac{p_1}{p_2})=-1.
   \end{array}\right.$
\item Assume that $(\frac{p_1}{p_2})=1$, so\\
 $\bullet$ If $\pi=1$, then   $\kappa_{\KK_4}=N_4$,  $\kappa_{ \KK_5}=N_5$,
 $\kappa_{ \KK_6}=N_6$  and  $\kappa_{ \KK_7}=N_7$.\\
  $\bullet$ Else\hspace{0.5cm}  $\kappa_{\KK_4}=N_7$,\quad $\kappa_{ \KK_5}=N_6$,\quad $\kappa_{ \KK_6}=N_5$ \quad and\quad  $\kappa_{ \KK_7}=N_4$.\\
\item Assume that $(\frac{p_1}{p_2})=-1$ and $q=1$, so\\
 $\kappa_{ \KK_4} =
\left\{\begin{array}{ll}
N_4 \hbox{ if } \pi=-1,\\
N_7 \hbox{ if } \pi=1.
      \end{array}
    \right.$
 $\kappa_{ \KK_5} =
\left\{\begin{array}{ll}
N_6 \hbox{ if } \pi=-1,\\
N_5 \hbox{ if } \pi=1.
      \end{array}
    \right.$\\
 $\kappa_{ \KK_6} =
\left\{\begin{array}{ll}
N_5 \hbox{ if } \pi=-1,\\
N_6 \hbox{ if } \pi=1.
      \end{array}
    \right.$
 $\kappa_{ \KK_7} =
\left\{\begin{array}{ll}
N_7 \hbox{ if } \pi=-1,\\
N_4 \hbox{ if } \pi=1.
      \end{array}
    \right.$\\
\item Assume  that $(\frac{p_1}{p_2})=-1$ and $q=2$, so\\
 $\kappa_{ \KK_4} =
\left\{\begin{array}{ll}
N_4 \hbox{ if } \pi=-1,\\
N_7 \hbox{ if } \pi=1.
      \end{array}
    \right.$
 $\kappa_{ \KK_5} =
\left\{\begin{array}{ll}
N_6 \hbox{ if } \pi=-1,\\
N_5 \hbox{ if } \pi=1.
      \end{array}
    \right.$\\
 $\kappa_{ \KK_6} =
\left\{\begin{array}{ll}
N_5 \hbox{ if } \pi=-1,\\
N_6 \hbox{ if } \pi=1.
      \end{array}
    \right.$
 $\kappa_{ \KK_7} =
\left\{\begin{array}{ll}
N_7 \hbox{ if } \pi=-1,\\
N_4 \hbox{ if } \pi=1.
      \end{array}
    \right.$
\end{enumerate}
\end{rema}
To check the tables inputs,   we use the following relations:\\
$\bullet$ $\sigma=\displaystyle\left(\frac{L/\KK_3}{\mathfrak{P}}\right)=
\displaystyle\left(\frac{L/\kk}{N_{\KK_3/\kk}(\mathfrak{P})}\right)=
\displaystyle\left(\frac{L/\kk}{\mathcal{H}_1\mathcal{H}_2}\right)$, since $N_{\KK_3/\kk}(\mathfrak{P})\sim \mathcal{H}_1\mathcal{H}_2$.\\
$\bullet$ $\tau=\displaystyle\left(\frac{L/\KK_3}{\mathfrak{A}}\right)=
\displaystyle\left(\frac{L/\kk}{N_{\KK_3/\kk}(\mathfrak{A})}\right)=
\displaystyle\left(\frac{L/\kk}{\mathcal{H}_0}\right)$, because $N_{\KK_3/\kk}(\mathfrak{A})\sim \mathcal{H}_0$.\\
$\bullet$ $\rho=\displaystyle\left(\frac{L/\kk}{\mathcal{H}_1}\right)$.\\
Recall that the Artin map $\phi$  induces the following commutative diagram:
\begin{figure}[H]
\xymatrix{
 &&&& \mathbf{C}l_2(\kk)\ar[d]_{j_{\KK_j/\kk}} \ar[r]^- \phi & G/G' \ar[d]^{{\rm V}_{G/G_j}}\\
 &&&& \mathbf{C}l_2(\KK_j) \ar[r]^-\phi & G_j/G_j' }
\end{figure}
\hspace{-0.4cm}the rows are isomorphisms and ${\rm V}_{G/G_j}: G/G' \longrightarrow G_j/G_j'$ is the group transfer map (Verlagerung) which has the following simple characterization when $G_j$ is of index $2$ in $G$. Let $G=G_j\cup zG_j$,  then
$${\rm V}_{G/G_j}(gG')=\left\{
        \begin{array}{ll}
          gz^{-1}gz.G_j'=g^2[g, z].G_j'  & \hbox{ if $g\in G_j$;} \\
          g^2G_j'  & \hbox{ if $g \notin G_j$.}
        \end{array}
      \right.$$
Thus $\kappa_{\KK_j}=\ker j_{\KK_j/\kk}$ is determined by $\ker{\rm V}_{G/ G_j}$.\\
\indent (a) Consider the extension  $\KK_1$; we know that  $G=\langle\sigma, \tau, \rho\rangle$ and, according to the table \ref{21},\\
$N_1=\left\{ \begin{array}{ll}
 \langle[\mathcal{H}_0\mathcal{H}_1], [\mathcal{H}_0\mathcal{H}_2]\rangle=\langle[\mathcal{H}_1\mathcal{H}_2], [\mathcal{H}_0\mathcal{H}_1]\rangle & \text{ if }\left(\frac{p_1}{p_2}\right)=1,\\
  \langle[\mathcal{H}_1], [\mathcal{H}_2]\rangle=\langle[\mathcal{H}_1\mathcal{H}_2], [\mathcal{H}_1]\rangle & \text{ if }\left(\frac{p_1}{p_2}\right)=-1.
 \end{array} \right.$\\
 Thus $G_1=\mathrm{Gal}(L/\KK_1)=\left\{ \begin{array}{ll}
 \langle\sigma, \tau\rho, G'\rangle=\langle\sigma, \tau\rho, \tau^2\rangle & \text{ if }\left(\frac{p_1}{p_2}\right)=1,\\
 \langle\sigma, \rho, G'\rangle=\langle\sigma, \rho, \tau^2\rangle=\langle\sigma, \rho\rangle & \text{ if }\left(\frac{p_1}{p_2}\right)=-1.
 \end{array} \right.$\\
This implies that $G/G_1=\langle\tau\rangle=\{1, \tau G_1\}$; as\\
 $[\tau\rho, \sigma]=[\rho, \sigma]=\left\{ \begin{array}{ll}
                                    \sigma^2 & \text{ if } q=1,\\
                                     \sigma^{-2^m+2}=\sigma^{-2} & \text{ if } q=2,\end{array} \right.$ and
$[\tau\rho, \tau^2]=\tau^4$.\\
So $G_1 '=\left\{ \begin{array}{ll}
\langle\tau^4, \sigma^2\rangle & \text{ if }\left(\frac{p_1}{p_2}\right)=1,\\
\langle \sigma^2\rangle & \text{ if }\left(\frac{p_1}{p_2}\right)=-1;
\end{array} \right.$
from which we deduce that\\ $\mathrm{Gal}(L/\KK_1)=G_1/G_1'\simeq\left\{ \begin{array}{ll}
 (2, 2, 2) & \text{ if }\left(\frac{p_1}{p_2}\right)=1,\\
 (2, 4)  & \text{ if }\left(\frac{p_1}{p_2}\right)=-1,
 \end{array} \right.$ since $(\tau\rho)^2=\rho^2\in G_1 '$.\\
Compute the kernel of ${\rm V}_{G/G_1}$.
\begin{enumerate}[\indent $\ast$]
\item ${\rm V}_{G/G_1}(\sigma G')=\sigma^2[\sigma,\rho]G_1'=\sigma^2\sigma^{-2}G_1'\text{ ou } \sigma^{4}G_1'=G_1'$.
\item ${\rm V}_{G/G_1}(\tau G')=\tau^2G_1'\neq G_1'$.
 \item ${\rm V}_{G/G_1}(\rho G')=\rho^2G_1'= G_1'$.
 \end{enumerate}
  Consequently $$ker{\rm V}_{G/G_1}=\langle\sigma G', \rho G'\rangle,$$ thus $$\kappa_{\KK_1}=\langle[\mathcal{H}_1\mathcal{H}_2], [\mathcal{H}_1]\rangle=\langle[\mathcal{H}_1], [\mathcal{H}_2]\rangle.$$\\
\indent (b) For   $\KK_2$, we proceed similarly, we get\\
 $G_2=\mathrm{Gal}(L/\KK_2)=\left\{ \begin{array}{ll}
 \langle\sigma, \rho, \tau^2\rangle &\text{ if }\left(\frac{p_1}{p_2}\right)=1,\\
 \langle\sigma, \tau\rho\rangle &\text{ if }\left(\frac{p_1}{p_2}\right)=-1;
  \end{array} \right.$ this implies that\\
$G_2 '=\left\{ \begin{array}{ll}
\langle\tau^4, \sigma^2\rangle & \text{ if }\left(\frac{p_1}{p_2}\right)=1,\\
\langle \sigma^2\rangle & \text{ if }\left(\frac{p_1}{p_2}\right)=-1;
\end{array} \right.$
from which we deduce that\\ $\mathrm{Gal}(L/\KK_2)=G_2/G_2'\simeq\left\{ \begin{array}{ll}
 (2, 2, 2) & \text{ if }\left(\frac{p_1}{p_2}\right)=1,\\
 (2, 4)  & \text{ if }\left(\frac{p_1}{p_2}\right)=-1.
 \end{array} \right.$\\
Thus
\begin{enumerate}[\indent $\ast$]
\item ${\rm V}_{G/G_2}(\sigma G')=\sigma^2[\sigma,\rho]G_2'=\sigma^2\sigma^{-2}G_2'\text{ or } \sigma^{4}G_2'=G_2'$.
\item ${\rm V}_{G/G_2}(\tau G')=\tau^2G_2'\neq G_2'$.
 \item ${\rm V}_{G/G_2}(\rho G')=\rho^2[\rho, \tau]G_2'=\tau^2 G_2'\neq G_2'$.
 \item ${\rm V}_{G/G_2}(\tau\rho G')=(\tau\rho)^2G_2'=\rho^2 G_2'= G_2'$.
 \end{enumerate}
  Therefore $$ker{\rm V}_{G/G_2}=\langle\sigma G', \tau\rho G'\rangle,$$ hence  $$\kappa_{\KK_2}=\langle[\mathcal{H}_1\mathcal{H}_2], [\mathcal{H}_0\mathcal{H}_1]\rangle=\langle[\mathcal{H}_0\mathcal{H}_1], [\mathcal{H}_0\mathcal{H}_2]\rangle.$$
 \indent (c) Take  $\KK_4=\kk(\sqrt{\pi_1\pi_3})$ and assume $\left(\frac{p_1}{p_2}\right)=1$.  We have to consider the following two cases:\\
 \indent  1\up{st} case:  Suppose that $\left(\frac{p_1}{p_2}\right)_4\left(\frac{p_2}{p_1}\right)_4=-1$, then  $n=1$, $m\geq3$, $q=1$, $N(\varepsilon_{p_1p_2})=1$ and, from  Lemma  \ref{4}, $\left(\frac{\pi_1}{\pi_3}\right)=-1$;  this in turn has two sub-cases:\\
  $(\alpha)$ If  $\left(\frac{1+i}{\pi_1}\right)\left(\frac{1+i}{\pi_3}\right)=1$, then Table  \ref{21} implies that \\ $N_{\KK_4/\kk}(\mathbf{C}l_2(\KK_4))=\langle[\mathcal{H}_0], [\mathcal{H}_2]\rangle$. As $\left(\frac{L/\kk}{\mathcal{H}_2}\right)=\left(\frac{L/\kk}{\mathcal{H}_1}\right)^{-1}
  \left(\frac{L/\kk}{\mathcal{H}_1\mathcal{H}_2}\right)=\rho^{-1}\sigma$,  so
   $$G_4=\mathrm{Gal}(L/\KK_4/)=\langle\tau, \rho^{-1}\sigma, G'\rangle.$$
    On the other hand, in this sub-case we have $\rho^{-1}\sigma\rho=\sigma^{-1}$, thus $\rho^{-1}\sigma=(\rho\sigma)^{-1}$, therefore  $$G_4=\mathrm{Gal}(L/\KK_4/)=\langle\tau, \rho\sigma, \sigma^2\rangle.$$  Since  $[\rho\sigma,\sigma^2]=\sigma^4$, $[\rho\sigma,\tau]=\tau^2$ and $[\sigma, \tau^2]=1$,  so $G_4'=\langle\tau^2, \sigma^4\rangle$. From which we deduce that  $G_4/G_4'\simeq(2, 2, 2)$, since $(\rho\sigma)^2=\rho^2=\sigma^{2^{m-1}}$. Moreover  $G/G_4=\langle\sigma\rangle$, then:
 \begin{enumerate}[\indent $\ast$]
\item ${\rm V}_{G/G_4}(\sigma G')=\sigma^2G_4'\neq G_4'.$
\item ${\rm V}_{G/G_4}(\tau G')=\tau^2[\tau,\sigma]G_4'= G_4'$.
 \item ${\rm V}_{G/G_4}(\rho G')=\rho^2G_4'=G_4'$.
 \end{enumerate}
   Consequently $$ker{\rm V}_{G/G_4}=\langle\tau G', \rho G'\rangle,$$ and thus $$\kappa_{\KK_4}=\langle[\mathcal{H}_0], [\mathcal{H}_1]\rangle.$$
  $(\beta)$ If  $\left(\frac{1+i}{\pi_1}\right)\left(\frac{1+i}{\pi_3}\right)=-1$, similarly we get  $$N_4=\langle[\mathcal{H}_2], [\mathcal{H}_0\mathcal{H}_1]\rangle=\langle[\mathcal{H}_0\mathcal{H}_1\mathcal{H}_2], [\mathcal{H}_0\mathcal{H}_1]\rangle,$$
   thus
   $$G_4=\mathrm{Gal}(L/\KK_4)=\langle\sigma\tau, \tau\rho, \sigma^2\rangle.$$
  As $[\tau\rho,\sigma\tau]=[\rho, \tau][\rho, \sigma]=(\sigma\tau)^2$ and $[\tau\rho,\sigma^2]=\sigma^4$, so $G_4'=\langle(\sigma\tau)^2\rangle$, since $\sigma^4=\sigma^4\tau^4=(\sigma\tau)^4$. It is clear that $(\sigma\tau)^2$,  $\sigma^4$ and  $(\tau\rho)^2=\rho^2=(\tau\sigma)^{2^{m-1}}$ are in $G_4'$; hence  $G_4/G_4'\simeq(2, 2, 2)$.\\
  Let us compute the kernel of ${\rm V}_{G/G_4}$. Since $G/G_4=\langle\sigma\rangle$, then:
 \begin{enumerate}[\indent $\ast$]
\item ${\rm V}_{G/G_4}(\sigma G')=\sigma^2G_4'\neq G_4'.$
\item ${\rm V}_{G/G_4}(\sigma\tau G')=(\sigma\tau)^2[\sigma\tau,\sigma]G_4'= G_4'$.
 \item ${\rm V}_{G/G_4}(\rho G')=\rho^2G_4'=G_4'$.
 \end{enumerate}
   This implies that $$ker{\rm V}_{G/G_4}=\langle\rho G', \sigma\tau G'\rangle,$$ hence $$\kappa_{\KK_2}=\langle[\mathcal{H}_1], [\mathcal{H}_1\mathcal{H}_2\mathcal{H}_0]\rangle=\langle[\mathcal{H}_0\mathcal{H}_2], [\mathcal{H}_1]\rangle.$$
\indent - 2\up{nd} case:  Suppose $\left(\frac{p_1}{p_2}\right)_4\left(\frac{p_2}{p_1}\right)_4=1$, then  $m=2$, $n\geq2$ and, from Lemma \ref{4}, $\left(\frac{\pi_1}{\pi_3}\right)=1$.
  Similarly, there are   two sub-cases to distinguish:\\
  $(\alpha)$ If  $\left(\frac{1+i}{\pi_1}\right)\left(\frac{1+i}{\pi_3}\right)=1$, then the  table \ref{21} implies that $N_4=\langle[\mathcal{H}_0], [\mathcal{H}_1]\rangle$, from which we deduce that
   $$G_4=\mathrm{Gal}(L/\KK_4/)=\langle\tau, \rho, G'\rangle=\langle\tau, \rho, \sigma^2\rangle;$$ as  $\rho^{2}=\sigma^{2}$ or $\rho^{2}=\sigma^{2}\tau^{2^n}$,  so $$G_4=\mathrm{Gal}(L/\KK_4/)=\langle\tau, \rho\rangle.$$  We know that $[\rho,\tau]=\tau^2$,  thus $G_4'=\langle\tau^2\rangle$, this in turn yields that   $G_4/G_4'\simeq(2, 4)$, since $\rho^4=1$. Moreover $G/G_4=\langle\sigma\rangle$, so:
 \begin{enumerate}[\indent $\ast$]
\item ${\rm V}_{G/G_4}(\sigma G')=\sigma^2G_4'\neq G_4'.$
\item ${\rm V}_{G/G_4}(\tau G')=\tau^2[\tau,\sigma]G_4'= G_4'$.
 \item ${\rm V}_{G/G_4}(\rho G')=\rho^2[\rho,\sigma]G_4'=\left\{
                                                                 \begin{array}{ll}
                                                                 \rho^2\sigma^2G_4'=\sigma^4G_4'=G_4', \text{ if } q=1,\\
                                                                 \rho^2\sigma^{-2}G_4'=\sigma^2\sigma^{-2}\tau^{2^n}G_4'=G_4', \text{ if } q=2;
                                                                 \end{array}\right.$
 \end{enumerate}
   Consequently $$ker{\rm V}_{G/G_4}=\langle\tau G', \rho G'\rangle,$$ and thus $$\kappa_{\KK_4}=\langle[\mathcal{H}_0], [\mathcal{H}_1]\rangle.$$
  $(\beta)$ If  $\left(\frac{1+i}{\pi_1}\right)\left(\frac{1+i}{\pi_3}\right)=-1$, then  we get   $$N_4=\langle[\mathcal{H}_1], [\mathcal{H}_0\mathcal{H}_2]\rangle=\langle[\mathcal{H}_1\mathcal{H}_2\mathcal{H}_0], [\mathcal{H}_1]\rangle,$$
   hence
   $$G_4=\mathrm{Gal}(L/\KK_4/)=\langle\sigma\tau, \rho, \sigma^2, \tau^2\rangle=\langle\sigma\tau, \rho\rangle.$$
  As $[\rho,\sigma\tau]=[\rho, \tau][\rho, \sigma]=
 \left\{
 \begin{array}{ll}
 (\sigma\tau)^2 & \text{ if } q=1,\\
 \sigma^{-2}\tau^2 & \text{ if } q=2,
 \end{array}\right.$\\
   so $G_4'= \left\{
 \begin{array}{ll}
 \langle(\sigma\tau)^2\rangle & \text{ if } q=1,\\
 \langle\sigma^{-2}\tau^2\rangle & \text{ if } q=2.
 \end{array}\right.$\\ Moreover, as $\sigma^{-2}\tau^2=\sigma^{-4}\sigma^2\tau^2=\tau^{-2^{n+1}}(\sigma\tau)^2=(\sigma\tau)^{2+2^{n+1}}$,  so $G_4'=
 \langle(\sigma\tau)^2\rangle$, thus  $G_4/G_4'\simeq(2, 4)$. On the other hand, $G/G_4=\langle\tau\rangle$, hence:
 \begin{enumerate}[\indent $\ast$]
\item ${\rm V}_{G/G_4}(\sigma G')=\sigma^2G_4'\neq G_4'.$
\item ${\rm V}_{G/G_4}(\tau G')=\tau^2G_4'\neq G_4'$.
 \item ${\rm V}_{G/G_4}(\rho G')=\rho^2[\rho, \tau]G_4'=\rho^2\tau^2G_4'=\\
  \left\{
 \begin{array}{ll}
 (\sigma\tau)^2G_4'=G_4' & \text{ if } q=1,\\
 \sigma^{2}\tau^{2^n}\tau^2G_4'=\sigma^{2}\tau^{-4}\tau^2G_4'=\sigma^{2}\tau^{-2}G_4'=G_4' & \text{ if } q=2;
 \end{array}\right.$\\
 since in the second case ($q=2$) we have  $\tau^{2^{n+2}}=1$, thus $\tau^{2^n}=\tau^{-4}$.
 \item ${\rm V}_{G/G_4}(\sigma\tau G')=(\sigma\tau)^2[\sigma\tau,\tau]G_4'= G_4'$.
 \end{enumerate}
   Therefore $$ker{\rm V}_{G/G_4}=\langle\rho G', \sigma\tau G'\rangle,$$ from which we deduce that $$\kappa_{\KK_4}=\langle[\mathcal{H}_1], [\mathcal{H}_1\mathcal{H}_2\mathcal{H}_0]\rangle=\langle[\mathcal{H}_0\mathcal{H}_2], [\mathcal{H}_1]\rangle.$$
 {\em Conclusion}:  Let $\KK_4=\kk(\sqrt{\pi_1\pi_3})$ and $G_4=\mathrm{Gal}(L/\KK_4)$. Then $\mathbf{C}l_2(\KK_4)$ is of type $(2, 2, 2)$ if  $\left(\frac{p_1}{p_2}\right)_4\left(\frac{p_2}{p_1}\right)_4=-1$, and of type $(2, 4)$ otherwise. Moreover
\begin{enumerate}[\indent(i)]
  \item If $\left(\frac{1+i}{\pi_1}\right)\left(\frac{1+i}{\pi_3}\right)=1$, then $ker{\rm V}_{G/G_4}=\langle\tau G', \rho G'\rangle$ and $\kappa_{\KK_4}=\langle[\mathcal{H}_0], [\mathcal{H}_1]\rangle$.
  \item If $\left(\frac{1+i}{\pi_1}\right)\left(\frac{1+i}{\pi_3}\right)=-1$, then $ker{\rm V}_{G/G_4}=\langle\rho G', \sigma\tau G'\rangle$, and  $\kappa_{\KK_4}=\langle[\mathcal{H}_0\mathcal{H}_2], [\mathcal{H}_1]\rangle$.
\end{enumerate}
Assume now that $\left(\frac{p_1}{p_2}\right)=-1$. We have also two cases to distinguish:\\
 \indent  1\up{st} case:  Suppose $q=1$, so  Lemma \ref{12} yields that  $n=1$,  $m\geq3$.  Then  we need  to consider two sub-cases:\\
  $\bullet$ If $\left(\frac{\pi_1}{\pi_3}\right)=1$, then $\left(\frac{1+i}{\pi_1}\right)\left(\frac{1+i}{\pi_3}\right)=1$, hence  $$N_4=\langle[\mathcal{H}_0], [\mathcal{H}_2]\rangle \text{ and } G_4=\mathrm{Gal}(L/\KK_4/)=\langle\tau, \rho\sigma, \sigma^2\rangle.$$
   So $G_4'=\langle\tau^2, \sigma^4\rangle$,  which involves that  $G_4/G_4'\simeq(2, 2, 2)$, since $(\rho\sigma)^2=\rho^2=\tau^2\sigma^{2^{m-1}}$. Moreover  $G/G_4=\langle\sigma\rangle$, then:
 \begin{enumerate}[\indent $\ast$]
\item ${\rm V}_{G/G_4}(\sigma G')=\sigma^2G_4'\neq G_4'.$
\item ${\rm V}_{G/G_4}(\tau G')=\tau^2[\tau,\sigma]G_4'= G_4'$.
 \item ${\rm V}_{G/G_4}(\rho G')=\rho^2G_4'=G_4'$.
 \end{enumerate}
   Consequently $$ker{\rm V}_{G/G_4}=\langle\tau G', \rho G'\rangle,$$ and thus $$\kappa_{\KK_2}=\langle[\mathcal{H}_0], [\mathcal{H}_1]\rangle.$$
 $\bullet$ If $\left(\frac{\pi_1}{\pi_3}\right)=-1$, then  $\left(\frac{1+i}{\pi_1}\right)\left(\frac{1+i}{\pi_3}\right)=-1$, similarly we get  $$N_4=\langle[\mathcal{H}_1], [\mathcal{H}_0\mathcal{H}_2]\rangle=\langle[\mathcal{H}_0\mathcal{H}_1\mathcal{H}_2], [\mathcal{H}_1]\rangle,$$
   thus
   $$G_4=\mathrm{Gal}(L/\KK_4/)=\langle\sigma\tau, \rho\rangle.$$
 So $G_4'=\langle(\sigma\tau)^2\rangle$,  hence  $G_4/G_4'\simeq(2, 4)$.
 As $G/G_4=\langle\tau\rangle$, thus
 \begin{enumerate}[\indent $\ast$]
\item ${\rm V}_{G/G_4}(\sigma G')=\sigma^2G_4'\neq G_4'.$
\item ${\rm V}_{G/G_4}(\sigma\tau G')=(\sigma\tau)^2[\sigma\tau,\sigma]G_4'= G_4'$.
 \item ${\rm V}_{G/G_4}(\rho G')=\rho^2[\rho,\tau]G_4'=\rho^2\tau^2G_4'=G_4'$.
 \end{enumerate}
   This implies that $$ker{\rm V}_{G/G_4}=\langle\rho G', \sigma\tau G'\rangle,$$ hence $$\kappa_{\KK_2}=\langle[\mathcal{H}_1], [\mathcal{H}_1\mathcal{H}_2\mathcal{H}_0]\rangle=\langle[\mathcal{H}_0\mathcal{H}_2], [\mathcal{H}_1]\rangle.$$
 {\em Conclusion}. Let   $\KK_4=\kk(\sqrt{\pi_1\pi_3})$;  assume that $\left(\frac{p_1}{p_2}\right)=-1$ and $q=1$.  If $\left(\frac{\pi_1}{\pi_3}\right)=-1$, then $\mathbf{C}l_2(\KK_4)$  is of type $(2, 4)$ and of type $(2, 2, 2)$ otherwise. Moreover
\begin{enumerate}[\indent(i)]
  \item If $\left(\frac{\pi_1}{\pi_3}\right)=1$, then $ker{\rm V}_{G/G_4}=\langle\rho G', \tau G'\rangle$ and $\kappa_{\KK_4}=\langle[\mathcal{H}_0], [\mathcal{H}_1]\rangle$.
  \item If $\left(\frac{\pi_1}{\pi_3}\right)=-1$, then $ker{\rm V}_{G/G_4}=\langle\rho G', \sigma\tau G'\rangle$ and   $\kappa_{\KK_4}=\langle[\mathcal{H}_1], [\mathcal{H}_0\mathcal{H}_2]\rangle$.
\end{enumerate}
  \indent  2\up{nd} case:  Suppose $q=2$, so,  according to Lemma \ref{12},  $n=1$ and $m=2$.  Then  we have two sub-cases  to consider :\\
  $\bullet$ If $\left(\frac{\pi_1}{\pi_3}\right)=1$, then $\left(\frac{1+i}{\pi_1}\right)\left(\frac{1+i}{\pi_3}\right)=-1$, hence  $$N_4=\langle[\mathcal{H}_2], [\mathcal{H}_0\mathcal{H}_1]\rangle \text{ and } G_4=\mathrm{Gal}(L/\KK_4/)=\langle\sigma\tau, \tau\rho, \sigma^2\rangle.$$
   So $G_4'=\langle\rho^2, \sigma^4\rangle$,  since $\rho^2=\tau^2\sigma^2$;  which implies that  $G_4/G_4'\simeq(2, 2, 2)$. Moreover  $G/G_4=\langle\rho\rangle$, then:
 \begin{enumerate}[\indent $\ast$]
\item ${\rm V}_{G/G_4}(\sigma G')=\sigma^2G_4'\neq G_4'.$
\item ${\rm V}_{G/G_4}(\tau G')=\tau^2G_4'\neq G_4'$.
 \item ${\rm V}_{G/G_4}(\rho G')=\rho^2G_4'=G_4'$.
\item ${\rm V}_{G/G_4}(\sigma\tau G')=(\sigma\tau)^2[\sigma\tau, \rho]G_4'=(\sigma\tau)^2(\sigma\tau)^{-2}\sigma^4G_4'=G_4'$.
 \end{enumerate}
   Consequently $$ker{\rm V}_{G/G_4}=\langle\rho G', \sigma\tau G'\rangle,$$  thus $$\kappa_{\KK_2}=\langle[\mathcal{H}_1], [\mathcal{H}_0\mathcal{H}_2]\rangle.$$
 $\bullet$ If $\left(\frac{\pi_1}{\pi_3}\right)=-1$, then  $\left(\frac{1+i}{\pi_1}\right)\left(\frac{1+i}{\pi_3}\right)=1$, similarly we get  $$N_4=\langle[\mathcal{H}_1], [\mathcal{H}_0]\rangle \text{ and } G_4=\mathrm{Gal}(L/\KK_4/)=\langle\tau, \rho\rangle.$$
 So $G_4'=\langle\tau^2\rangle$,  hence  $G_4/G_4'\simeq(2, 4)$.
 As $G/G_4=\langle\sigma\rangle$, thus
 \begin{enumerate}[\indent $\ast$]
\item ${\rm V}_{G/G_4}(\sigma G')=\sigma^2G_4'\neq G_4'.$
\item ${\rm V}_{G/G_4}(\tau G')=\tau^2G_4'= G_4'$.
 \item ${\rm V}_{G/G_4}(\rho G')=\rho^2[\rho,\sigma]G_4'=\rho^2\sigma^{-2}G_4'=G_4'$.
 \end{enumerate}
   This implies that $$ker{\rm V}_{G/G_4}=\langle\rho G', \tau G'\rangle,$$ hence $$\kappa_{\KK_2}=\langle[\mathcal{H}_1], [\mathcal{H}_0]\rangle.$$
 {\em Conclusion}. Let   $\KK_4=\kk(\sqrt{\pi_1\pi_3})$;  assume that $\left(\frac{p_1}{p_2}\right)=-1$ and $q=2$. Then $\mathbf{C}l_2(\KK_4)$  is of type $(2, 4)$ if $\left(\frac{\pi_1}{\pi_3}\right)=-1$,   and of type $(2, 2, 2)$ otherwise. Moreover
\begin{enumerate}[\indent(i)]
  \item If $\left(\frac{\pi_1}{\pi_3}\right)=-1$, then $ker{\rm V}_{G/G_4}=\langle\rho G', \tau G'\rangle$ and   $\kappa_{\KK_4}=\langle[\mathcal{H}_0], [\mathcal{H}_1]\rangle$.
  \item Else $ker{\rm V}_{G/G_4}=\langle\rho G', \sigma\tau G'\rangle$ and $\kappa_{\KK_4}=\langle[\mathcal{H}_1], [\mathcal{H}_0\mathcal{H}_2]\rangle$.
\end{enumerate}
Proceeding similarly we show the other tables inputs.
\subsubsection{\bf Capitulation kernels $\kappa_{\LL_j}$ and $\mathrm{Gal}(\L/\LL_j)$} From the subsection \ref{17}, we deduce that $\kappa_{\LL_j}=\mathbf{C}l_2(\kk)$. In what follows, we compute the Galois groups $\mathcal{G}_j=\mathrm{Gal}(\L/\LL_j)$, their derived groups $\mathcal{G}_j'$ and the abelian type invariants of $\mathbf{C}l_2(\LL_j)$. The results are summarized in the following tables; note that the left hand side of columns (if it exists) refers to the case $(\frac{1+i}{\pi_1})(\frac{1+i}{\pi_3})=1$, the right hand side to the case $(\frac{1+i}{\pi_1})(\frac{1+i}{\pi_3})=-1$.\\  Put
 $\left\{
 \begin{array}{ll}
 \alpha=\min(n, m) \text{ and } \beta=\max(n+1, m+1),\\
 a=\min(n, m-1) \text{ and } b=\max(n+1, m),\\
 \pi=(\frac{\pi_1}{\pi_3}),\\
 B=(\frac{1+i}{\pi_1})(\frac{1+i}{\pi_3}).
 \end{array}
 \right.$
 \tiny
 \begin{longtable}{|c c | c | c | c |}
\caption{ \small{Invariants of $\mathbf{C}l_2(\LL_j)$ for the case $\left(\frac{p_1}{p_2}\right)=1$}}\\
\hline
 $\LL_j$  & & $\mathcal{G}_j$ & $\mathcal{G}_j'$ & $\mathbf{C}l_2(\LL_j)$\\
\hline
\endfirsthead
\hline
 $\LL_j$ & &$\mathcal{G}_j$ & $\mathcal{G}_j$ & $\mathbf{C}l_2(\LL_j)$\\
\hline
\endhead
 & && & $(2^n, 2^{m})$ if $q=1$\\[-1ex]
 \raisebox{2ex}{$\LL_1$} && \raisebox{2ex}{$\langle \tau^2, \sigma\rangle$} & \raisebox{2ex}{$\langle1\rangle$} &
  $(2^{\alpha}, 2^{\beta})$ if $q=2$\\[1ex] \hline
  & $\pi=-1$ &  $\langle\sigma\tau\rho, \sigma^2, \tau^2\rangle$ $\langle\tau\rho, \sigma^2, \tau^2\rangle$  & $\langle\sigma^4, \tau^4\rangle$ & $(2, 2, 2)$\\[-1ex]
  \raisebox{2ex}{$\LL_2$}
 &$\pi=1$ &  $\langle\tau\rho, \tau^2\rangle$ $\langle\sigma\tau\rho, \tau^2\rangle$   & $\langle \tau^4\rangle$ & $(2, 4)$\\[1ex]\hline
 &$\pi=-1$ &  $\langle\tau\rho, \sigma^2, \tau^2\rangle$ $\langle\sigma\tau\rho, \sigma^2, \tau^2\rangle$  & $\langle\sigma^4, \tau^4\rangle$ & $(2, 2, 2)$\\[-1ex]
\raisebox{2ex}{$\LL_3$}
& $\pi=1$& $\langle\sigma\tau\rho,  \tau^2\rangle$ $\langle\tau\rho,  \tau^2\rangle$  & $\langle \tau^4\rangle$ & $(2, 4)$\\ [1ex]\hline
  & $\pi=-1$&   $\langle \sigma\rho, \sigma^2, \tau^2\rangle$  & $\langle\sigma^4, \tau^4\rangle$ & $(2, 2, 2)$\\[-1ex]
 \raisebox{2ex}{$\LL_4$}
    &$\pi=1$     &    $\langle \rho,  \tau^2\rangle$  & $\langle \tau^4\rangle$ & $(2, 4)$ \\[1ex]\hline
  & $\pi=-1$&  $\langle \rho, \sigma^2, \tau^2\rangle$  & $\langle\sigma^4, \tau^4\rangle$ & $(2, 2, 2)$\\[-1ex]
\raisebox{1.7ex}{$\LL_5$}
  &$\pi=1$&   $\langle \rho\sigma,  \tau^2\rangle$  & $\langle \tau^4\rangle$ & $(2, 4)$\\[1ex]\hline
 & $B=-1$ & $\langle\sigma\tau, \sigma^2\rangle$  & $\langle1\rangle$ & $\left\{\begin{array}{ll}(2^{a}, 2^{b}) \text{ if } q=1\\
                          (2, 2^{n+2}) \text{ if } q=2 \end{array}\right.$ \\[-1ex]
\raisebox{2ex}{$\LL_6$}
      & $B=1$  & $\langle\tau, \sigma^2\rangle$   & $\langle1\rangle$ &
          $\left\{\begin{array}{ll}(2^{m-1}, 2^{n+1}) \text{ if } q=1\\
                          (2, 2^{n+2}) \text{ if } q=2 \end{array}\right.$\\[1ex]\hline
 &$B=-1$ &  $\langle\tau, \sigma^2\rangle$     & $\langle1\rangle$ &
   $\left\{\begin{array}{ll}(2^{m-1}, 2^{n+1}) \text{ if } q=1\\
                          (2, 2^{n+2}) \text{ if } q=2 \end{array}\right.$\\[-1ex]
\raisebox{2ex}{$\LL_7$}
   & $B=1$  &  $\langle\sigma\tau, \sigma^2\rangle$   & $\langle1\rangle$ &
           $\left\{\begin{array}{ll}(2^{a}, 2^{b}) \text{ if } q=1\\
                          (2, 2^{n+2}) \text{ if } q=2 \end{array}\right.$\\[1ex]\hline
\end{longtable}
\normalsize
For the following table,  the left hand side of columns (if it exists) refers to the case $(\frac{\pi_1}{\pi_3})=-1$, the right hand side refers to the case $(\frac{\pi_1}{\pi_2})=1$. Recall that if  $(\frac{p_1}{p_2})=-1$, then $n=1$ and
 $\left\{\begin{array}{ll}
 m=2 & \text{ if } q=2,\\
m\geq2 & \text{ if } q=1
\end{array}\right.$
\tiny
 \begin{longtable}{|c c | c | c | c |}
\caption{\large{ Invariants of $\mathbf{C}l_2(\LL_j)$ for the case $\left(\frac{p_1}{p_2}\right)=-1$}}\\
\hline
 $\LL_j$  & & $\mathcal{G}_j$ & $\mathcal{G}_j'$ & $\mathbf{C}l_2(\LL_j)$\\
\hline
\endfirsthead
\hline
 $\LL_j$ & &$\mathcal{G}_j$ & $\mathcal{G}_j$ & $\mathbf{C}l_2(\LL_j)$\\
\hline
\endhead
 & && & $(2^n, 2^{m})$ if $q=1$\\[-1ex]
 \raisebox{2ex}{$\LL_1$} && \raisebox{2ex}{$\langle \tau^2, \sigma\rangle$} & \raisebox{2ex}{$\langle1\rangle$} &
  $(2^{n}, 2^{m+1})$ if $q=2$\\[1ex] \hline
 $\LL_2$ & &  $\langle\rho, \sigma^2\rangle$ $\langle\rho\sigma, \sigma^2\rangle$ & $\langle \sigma^4\rangle$ & $(2, 4)$\\\hline
$\LL_3$ & &  $\langle\rho\sigma, \sigma^2\rangle$ $\langle\rho, \sigma^2\rangle$ & $\langle \sigma^4\rangle$ & $(2, 4)$\\\hline
 &$q=1$ &  $\langle\sigma\tau\rho, \sigma^2\rangle$  &  & \\[-1ex]
\raisebox{2ex}{$\LL_4$}
  &$q=2$ &$\langle\tau\rho, \sigma^2\rangle$ & \raisebox{2ex}{$\langle \sigma^4\rangle$}& \raisebox{2ex}{$(2, 4)$}\\\hline
 &$q=1$ &  $\langle\tau\rho, \sigma^2\rangle$  &  & \\[-1ex]
\raisebox{2ex}{$\LL_5$}
  &$q=2$ &$\langle\sigma\tau\rho, \sigma^2\rangle$ & \raisebox{2ex}{$\langle \sigma^4\rangle$}& \raisebox{2ex}{$(2, 4)$}\\\hline
 & $q=1$ & $\langle\sigma\tau,  \tau^2\rangle$ $\langle \tau,  \sigma^2\rangle$  & $\langle1\rangle$ &
                 $(2, 2^{m})$ $(4, 2^{m-1})$ \\[-1ex]
\raisebox{2ex}{$\LL_6$}
      & $q=2$  & $\langle\tau, \sigma^2\rangle$  $\langle\sigma\tau, \sigma^2\rangle$ & $\langle1\rangle$ &
          $(2, 8)$ $(4, 4)$\\[1ex]\hline
 &$q=1$ &  $\langle\tau, \sigma^2\rangle$ $\langle\sigma\tau, \sigma^2\rangle$    & $\langle1\rangle$ &
   $(4, 2^{m-1})$ $(2, 2^{m})$\\[-1ex]
\raisebox{2ex}{$\LL_7$}
   & $q=2$  &  $\langle\sigma\tau, \sigma^2\rangle$ $\langle\tau, \sigma^2\rangle$  & $\langle1\rangle$ &
           $(4, 4)$ $(2, 8)$\\[1ex]\hline
\end{longtable}
\normalsize
Check the entries in  some cases.\\
\indent $\ast$ Take $\mathbb{L}_1=\k=\KK_1.\KK_2.\KK_3$.  Since $\mathrm{Gal}(L/\LL_1)=\mathcal{G}_1=G_1\cap G_2$, then \\
 $ \mathcal{G}_1=
\langle\sigma, \tau\rho, \tau^2\rangle\cap\langle\sigma, \rho, \tau^2\rangle=\langle\sigma, \tau^2\rangle, $
 thus $\mathcal{G}_1'=\langle1\rangle$. As \\
  $\left\{\begin{array}{ll}
\sigma^{2^m}=\tau^{2^{n+1}}=1  & \hbox{ if } q=1,\\
\sigma^{2^m}=\tau^{2^{n+1}} \hbox{ and } \sigma^{2^{m+1}}=\tau^{2^{n+2}}=1  & \hbox{ if } q=2,
\end{array}\right.$\\
  so $\mathbf{C}l_2(\LL_1)\simeq\left\{\begin{array}{ll}
(2^n, 2^m)  & \hbox{ if } q=1,\\
(2^{\min(n,m)}, 2^{\max(n+1,m+1)})  & \hbox{ if } q=2.
\end{array}\right.$\\
\indent $\ast$ Take
 $\mathbb{L}_2=\KK_1.\KK_4.\KK_6$ and assume that $(\frac{p_1}{p_2})=1$, then $\mathcal{G}_2=\mathrm{Gal}(L/\LL_2)=G_1\cap G_4\cap G_6$. There are two  cases to distinguish: \\
 - 1\up{st} case: If $\left(\frac{p_1}{p_2}\right)_4\left(\frac{p_2}{p_1}\right)_4=\left(\frac{\pi_1}{\pi_3}\right)=-1$, then $q=1$ and\\
$\mathcal{G}_2=
 \left\{\begin{array}{ll}
\langle\sigma, \tau\rho, \tau^2\rangle \cap \langle\sigma\tau, \rho, \tau^2\rangle=\langle\sigma\tau\rho, \sigma^2, \tau^2\rangle  & \hbox{ if } (\frac{1+i}{\pi_1})(\frac{1+i}{\pi_3})=1,\\
\langle\sigma, \tau\rho, \tau^2\rangle \cap \langle\sigma\tau, \tau\rho, \sigma^2\rangle=\langle\tau\rho, \sigma^2, \tau^2\rangle  & \hbox{ if } (\frac{1+i}{\pi_1})(\frac{1+i}{\pi_3})=-1,
\end{array}\right.$\\
  thus $\mathcal{G}_2'=\langle\sigma^4, \tau^4\rangle$. On the other hand, as in this case $(\sigma\tau\rho)^2=(\tau\rho)^2=\rho^2=\sigma^{2^{m-1}}$, so   $\mathbf{C}l_2(\LL_2)\simeq(2, 2, 2)$.\\
  - 2\up{nd} case: If $\left(\frac{p_1}{p_2}\right)_4\left(\frac{p_2}{p_1}\right)_4=\left(\frac{\pi_1}{\pi_3}\right)=1$, then\\
$\mathcal{G}_2=
 \left\{\begin{array}{ll}
\langle\sigma, \tau\rho, \tau^2\rangle \cap \langle\tau, \rho, \sigma^2\rangle=\langle\tau\rho, \sigma^2, \tau^2\rangle  & \hbox{ if } (\frac{1+i}{\pi_1})(\frac{1+i}{\pi_3})=1,\\
\langle\sigma, \tau\rho, \tau^2\rangle \cap \langle\sigma\tau, \rho, \sigma^2\rangle=\langle\sigma\tau\rho, \sigma^2, \tau^2\rangle  & \hbox{ if } (\frac{1+i}{\pi_1})(\frac{1+i}{\pi_3})=-1,
\end{array}\right.$\\
as, in this case, $(\tau\rho)^2=\rho^2$ and $(\sigma\tau\rho)^2=\left\{\begin{array}{ll}
                                                               \rho^2=\sigma^{2} & \text{ if } q=1,\\
                                                               \rho^2\sigma^4=\sigma^2\tau^{-2^n} & \text{ if } q=2;
                                                               \end{array}\right.$\\ since if $q=2$, we have $\sigma^4=\tau^{2^{n+1}}$ and $\rho^2=\sigma^2\tau^{2^n}$,\\ so
$\mathcal{G}_2=
 \left\{\begin{array}{ll}
\langle\tau\rho,  \tau^2\rangle  & \hbox{ if } (\frac{1+i}{\pi_1})(\frac{1+i}{\pi_3})=1,\\
\langle\sigma\tau\rho,  \tau^2\rangle  & \hbox{ if } (\frac{1+i}{\pi_1})(\frac{1+i}{\pi_3})=-1,
\end{array}\right.$\\
 we infer that $\mathcal{G}_2'=\langle\tau^4\rangle$. From which we deduce that   $\mathbf{C}l_2(\LL_2)\simeq(2, 4)$, since $(\sigma\tau\rho)^4=(\tau\rho)^4=1$.\\
 Assume now that $(\frac{p_1}{p_2})=-1$, then
$\mathcal{G}_2=
 \left\{\begin{array}{ll}
\langle\rho, \sigma^2\rangle  & \hbox{ if } \left(\frac{\pi_1}{\pi_3}\right)=-1,\\
\langle\rho\sigma,  \sigma^2\rangle  & \hbox{ if } \left(\frac{\pi_1}{\pi_3}\right)=1.
\end{array}\right.$\\
  Thus $\mathcal{G}_2'=\langle\sigma^4\rangle$, hence   $\mathbf{C}l_2(\LL_2)\simeq(2, 4)$.\\
 \indent $\ast$ Finally, we take
 $\mathbb{L}_6=\KK_3.\KK_4.\KK_7$ and we assume that $(\frac{p_1}{p_2})=1$, then $\mathrm{Gal}(L/\LL_6)=\mathcal{G}_6=G_3\cap G_4\cap G_7$, which yields that \\
  $\mathcal{G}_6=
 \left\{\begin{array}{ll}
\langle\tau, \sigma\rangle \cap \langle\tau, \rho, \sigma^2\rangle=\langle\tau, \sigma^2\rangle  & \hbox{ if }  (\frac{1+i}{\pi_1})(\frac{1+i}{\pi_3})=1,\\
\langle\tau, \sigma\rangle \cap \langle\sigma\tau, \rho, \sigma^2\rangle=\langle\sigma\tau, \sigma^2\rangle=\langle\sigma\tau, \tau^2\rangle  & \hbox{ if }  (\frac{1+i}{\pi_1})(\frac{1+i}{\pi_3})=-1,
\end{array}\right.$\\
 therefore  $\mathcal{G}_6'=\langle1\rangle$.\\ As
   $\left\{\begin{array}{ll}
\sigma^{2^m}=\tau^{2^{n+1}}=1  & \hbox{ if } q=1,\\
\sigma^{2^m}=\tau^{2^{n+1}} \hbox{ and } \sigma^{2^{m+1}}=\tau^{2^{n+2}}=1  & \hbox{ if } q=2,
\end{array}\right.$ so\\
- 1\up{st} case: If $q=1$, then
  $\mathbf{C}l_2(\LL_6)$ is of type\\
 $\left\{\begin{array}{ll}
(2^{\min(n+1, m-1)}, 2^{\max(n+1, m-1)})=(2^{n+1}, 2^{m-1})  & \hbox{ if }  (\frac{1+i}{\pi_1})(\frac{1+i}{\pi_3})=1,\\
(2^{\min(n,m-1)}, 2^{\max(n+1,m)})   & \hbox{ if }  (\frac{1+i}{\pi_1})(\frac{1+i}{\pi_3})=-1,
\end{array}\right.$\\
- 2\up{nd} case: If $q=2$, then $m=2$, $n\geq2$ and
  $\mathbf{C}l_2(\LL_6)$ is of type\\
 $\left\{\begin{array}{ll}
 (2^{\min(n+1,m-1)}, 2^{\max(n+2,m+1)})=(2, 2^{n+2})  & \hbox{ if }  (\frac{1+i}{\pi_1})(\frac{1+i}{\pi_3})=1,\\
(2^{\min(n,m-1)}, 2^{\max(n+2,m+1)})=(2, 2^{n+2})  & \hbox{ if }  (\frac{1+i}{\pi_1})(\frac{1+i}{\pi_3})=-1,
\end{array}\right.$\\
 \indent Assume that $(\frac{p_1}{p_2})=-1$, so\\
 - 1\up{st} case: If $q=1$, then
  $\mathcal{G}_6=
 \left\{\begin{array}{ll}
\langle\tau\sigma, \tau^2\rangle  & \hbox{ if }  (\frac{\pi_1}{\pi_3})=-1,\\
\langle\tau, \sigma^2\rangle  & \hbox{ if }  (\frac{\pi_1}{\pi_3})=1,
\end{array}\right.$\\
 therefore  $\mathcal{G}_6'=\langle1\rangle$. So
  $\mathbf{C}l_2(\LL_6)\simeq\left\{\begin{array}{ll}
(2, 2^{m})  & \hbox{ if }  (\frac{\pi_1}{\pi_3})=-1,\\
(4, 2^{m-1})   & \hbox{ if }  (\frac{\pi_1}{\pi_3})=1,
\end{array}\right.$\\
- 2\up{nd} case: If $q=2$, then $m=2$, $n=1$ and
$\mathcal{G}_6=
 \left\{\begin{array}{ll}
\langle\tau, \sigma^2\rangle  & \hbox{ if }  (\frac{\pi_1}{\pi_3})=-1,\\
 \langle\tau\sigma, \tau^2\rangle  & \hbox{ if }  (\frac{\pi_1}{\pi_3})=1.
\end{array}\right.$\\ Hence
  $\mathbf{C}l_2(\LL_6)\simeq\left\{\begin{array}{ll}
 (2, 8)=(2, 2^{n+2})  & \hbox{ if }  (\frac{\pi_1}{\pi_3})=-1,\\
(4, 4)  & \hbox{ if }  (\frac{\pi_1}{\pi_3})=1,
\end{array}\right.$\\
The other  tables entries are checked similarly.
\section{Numerical examples}
Table \ref{29}  gives the structure of the class group $\mathbf{C}l(\kk)$ of the bicyclic biquadratic field $\kk=\QQ(\sqrt{2p_1p_2}, i)$, its  discriminant $disc(\kk)$, the structures of the class groups of its two quadratic subfields $k_0$, $\overline{k}_0$ and the coclass of $G=\mathrm{Gal}(\L/\kk)$. Tables \ref{30} and \ref{31} give the structures of the class groups  $\mathbf{C}l(\KK_j)$.  Tables \ref{32} and \ref{33} give the structures of the class groups  $\mathbf{C}l(\LL_j)$.  Finally, Tables \ref{34} and \ref{35} give the structures of the class groups  of $\mathbf{C}l(\KK_j)$ and  $\mathbf{C}l(\LL_j)$ for the case $(\frac{p_1}{p_2})=-1$. Note that  $\pi=(\frac{\pi_1}{\pi_3})$ and  $b=\left(\frac{1+i}{\pi_1}\right)\left(\frac{1+i}{\pi_3}\right)$.
 Computation are made  using PARI/GP \cite{GP-11}.
\tiny
\begin{longtable}{| c | c | c | c | c | c | c | c | c |}
\caption{\small{ Invariants of $\kk$}\label{29}}\\
\hline
 $d$  $ = p_1.p_2.q$ & $q$ & $(\frac{p_1}{p_2})$ & $m$, $n$
             & $\mathbf{C}l_2(k_0)$ & $\mathbf{C}l_2(\overline{k}_0)$ & $\mathbf{C}l_2(\kk)$ & $disc(\kk)$ & $cc(G)$ \\
\hline
\endfirsthead
\hline
 $d$  $ = p_1.p_2.q$ & $q$ & $(\frac{p_1}{p_2})$ & $m$, $n$
             & $\mathbf{C}l_2(k_0)$ & $\mathbf{C}l_2(\overline{k}_0)$ & $\mathbf{C}l_2(\kk)$ & $disc(\kk)$ & $cc(G)$ \\
\hline
\endhead
$130 = 2.5.13 $ & $2 $ & $-1 $ & $2$, $1$ & $(2, 2) $ & $ (2, 2) $ & $ (2 , 2, 2) $ & $1081600 $ & $ 3$\\
$290 = 2.5.29 $ & $1 $ & $1 $ & $2$, $2$ & $(2, 2) $ & $ (10, 2) $ & $ (10 , 2, 2) $ & $5382400 $ & $ 3$\\
$370 = 2.5.37 $ & $1 $ & $-1 $ & $3$, $1$ & $(2, 2) $ & $ (6, 2) $ & $ (6 , 2, 2) $ & $8761600 $ & $ 3$\\
$754 = 2.13.29 $ & $1 $ & $1 $ & $3$, $1$ & $(2, 2) $ & $ (10, 2) $ & $ (10 , 2, 2) $ & $36385024 $ & $ 3$\\
$3922 = 2.53.37 $ & $1 $ & $1 $ & $4$, $1$ & $(2, 2) $ & $ (10, 2) $ & $ (10 , 2, 2) $ & $984453376 $ & $ 3$\\
$4610 = 2.461.5 $ & $2 $ & $1 $ & $2$, $4$ & $(2, 2) $ & $ (26, 2) $ & $ (26 , 2, 2) $ & $1360134400 $ & $ 3$\\
$5122 = 2.197.13 $ & $1 $ & $-1 $ & $5$, $1$ & $(2, 2) $ & $ (14, 2) $ & $ (14 , 2, 2) $ & $1679032576 $ & $ 3$\\
$5410 = 2.5.541 $ & $1 $ & $1 $ & $2$, $3$ & $(2, 2) $ & $ (22, 2) $ & $ (22 , 2, 2) $ & $1873158400 $ & $ 3$\\
\hline
\end{longtable}
\tiny
\begin{longtable}{| p{0.4in} | p{0.02in} | p{0.25in} |p{0.4in} |p{0.4in} | p{0.4in} | p{0.4in} |p{0.4in} | p{0.4in} | p{0.4in} |}
\caption{\small{ Invariants of $\KK_j$ for the case $(\frac{p_1}{p_2})=1$ and $(\frac{\pi_1}{\pi_3})=-1$}\label{30}}\\
\hline
 $2.p_1.p_2$ & $q$ & $m,$ $n$
             & $\mathbf{C}l(\KK_1)$ & $\mathbf{C}l(\KK_2)$ & $\mathbf{C}l(\KK_3)$ & $\mathbf{C}l(\KK_4)$
			   & $\mathbf{C}l(\KK_5)$ & $\mathbf{C}l(\KK_6)$ & $\mathbf{C}l(\KK_7)$\\
\hline
\endfirsthead
\hline
 $2.p_1.p_2$ & $q$ & $m,$ $n$
             & $\mathbf{C}l(\KK_1)$ & $\mathbf{C}l(\KK_2)$ & $\mathbf{C}l(\KK_3)$ & $\mathbf{C}l(\KK_4)$
  			  & $\mathbf{C}l(\KK_5)$ & $\mathbf{C}l(\KK_6)$ & $\mathbf{C}l(\KK_7)$\\
\hline
\endhead
$ 2.5.269 $ & $1 $ & $3$, $1 $ & $ (30, 10, 2) $ & $ (330, 2, 2) $ & $ (120 , 12)
 $ & $ (30, 2, 2) $ & $ (30, 2, 2) $ & $ (30, 2, 2) $ & $ (30, 2, 2) $\\
$ 2.53.29 $ & $1 $ & $3$, $1 $ & $ (78, 2, 2) $ & $ (78, 6, 2) $ & $ (104 , 4)
 $ & $ (130, 2, 2) $ & $ (26, 2, 2) $ & $ (26, 2, 2) $ & $ (130, 2, 2) $\\
$2.5.389 $ & $1 $ & $3$, $1 $ & $ (42, 14, 2) $ & $ (462, 2, 2) $ & $ (168 , 4)
 $ & $ (42, 2, 2) $ & $ (42, 6, 2) $ & $ (42, 6, 2) $ & $ (42, 2, 2) $\\
$2.53.37 $ & $1 $ & $4$, $1 $ & $ (30, 10, 2) $ & $ (30, 2, 2) $ & $ (80 , 4)
 $ & $ (10, 2, 2) $ & $ (30, 2, 2) $ & $ (30, 2, 2) $ & $ (10, 2, 2) $\\
$ 2.13.157 $ & $1 $ & $4$, $1 $ & $ (26, 26, 2) $ & $ (78, 6, 2) $ & $ (208 , 4)
 $ & $ (26, 2, 2) $ & $ (26, 2, 2) $ & $ (26, 2, 2) $ & $ (26, 2, 2) $\\
$ 2.13.269 $ & $1 $ & $4$, $1 $ & $ (90, 10, 2) $ & $ (990, 6, 2) $ & $ (720 , 4)
 $ & $ (90, 2, 2) $ & $ (90, 2, 2) $ & $ (90, 2, 2) $ & $ (90, 2, 2) $\\
\hline
\end{longtable}
\normalsize
\tiny
\begin{longtable}{| c | c | c | c | c |c | c | c | c| c |}
\caption{\small{ Invariants of $\KK_j$ for the case $(\frac{p_1}{p_2})=1$ and $(\frac{\pi_1}{\pi_3})=1$}\label{31}}\\
\hline
 $d$  $ = 2.p_1.p_2$ & $q$ & $m,$ $n$
             & $\mathbf{C}l(\KK_1)$ & $\mathbf{C}l(\KK_2)$ & $\mathbf{C}l(\KK_3)$ & $\mathbf{C}l(\KK_4)$
			   & $\mathbf{C}l(\KK_5)$ & $\mathbf{C}l(\KK_6)$ & $\mathbf{C}l(\KK_7)$\\
\hline
\endfirsthead
\hline
 $d$  $ = 2.p_1.p_2$ & $q$ & $m,$ $n$
             & $\mathbf{C}l(\KK_1)$ & $\mathbf{C}l(\KK_2)$ & $\mathbf{C}l(\KK_3)$ & $\mathbf{C}l(\KK_4)$
  			  & $\mathbf{C}l(\KK_5)$ & $\mathbf{C}l(\KK_6)$ & $\mathbf{C}l(\KK_7)$\\
\hline
\endhead
$4498 = 2.13.173 $ & $1 $ & $2$, $2 $ & $ (210, 2, 2) $ & $ (42, 14, 2) $ & $ (280 , 4)
 $ & $ (28, 2) $ & $ (28, 2) $ & $ (28, 2) $ & $ (28, 2) $\\
$4610 = 2.461.5 $ & $2 $ & $2$, $4 $ & $ (390, 2, 2) $ & $ (234, 2, 2) $ & $ (2496 , 4)
 $ & $ (780, 2) $ & $ (260, 2) $ & $ (260, 2) $ & $ (780, 2) $\\
$5090 = 2.5.509 $ & $2 $ & $2$, $2 $ & $ (234, 2, 2) $ & $ (390, 2, 2) $ & $ (624 , 4)
 $ & $ (52, 2) $ & $ (156, 2) $ & $ (156, 2) $ & $ (52, 2) $\\
$5410 = 2.541.5 $ & $1 $ & $2$, $3 $ & $ (110, 2, 2) $ & $ (22, 22, 2) $ & $ (1584 , 4)
 $ & $ (44, 2) $ & $ (44, 2) $ & $ (44, 2) $ & $ (44, 2) $\\
$6322 = 2.29.109 $ & $1 $ & $2$, $2 $ & $ (90, 6, 2) $ & $ (18, 6, 2) $ & $ (504 , 4)
 $ & $ (180, 2) $ & $ (36, 6) $ & $ (36, 6) $ & $ (180, 2) $\\
$7090 = 2.709.5 $ & $2 $ & $2$, $2 $ & $ (130, 2, 2) $ & $ (442, 2, 2) $ & $ (1872 , 4)
 $ & $ (52, 2) $ & $ (52, 2) $ & $ (52, 2) $ & $ (52, 2) $\\
\hline
\end{longtable}
\normalsize
\tiny
\begin{longtable}{| c  c  c  c  c c  c  c  c c  c |}
\caption{\small{ Invariants of $\LL_j$ for $(\frac{p_1}{p_2})=1$ and $(\frac{\pi_1}{\pi_3})=1$,} $m$ is always $2$ \label{32}}\\
\hline
 $d$   & $q$ & $b$ &  $n$
             & $\mathbf{C}l(\LL_1)$ & $\mathbf{C}l(\LL_2)$ & $\mathbf{C}l(\LL_3)$ & $\mathbf{C}l(\LL_4)$
    			 & $\mathbf{C}l(\LL_5)$ & $\mathbf{C}l(\LL_6)$ & $\mathbf{C}l(\LL_7)$\\
\hline
\endfirsthead
\hline
 $d$   & $q$ & $b$ &  $n$
             & $\mathbf{C}l(\LL_1)$ & $\mathbf{C}l(\LL_2)$ & $\mathbf{C}l(\LL_3)$ & $\mathbf{C}l(\LL_4)$
    			 & $\mathbf{C}l(\LL_5)$ & $\mathbf{C}l(\LL_6)$ & $\mathbf{C}l(\LL_7)$\\
\hline
\endhead
$17090 $ & $2 $ & $-1 $ &  $2 $ & $ (9240, 420) $ & $ (4620, 2) $ & $ (4620 , 2)
 $ & $ (420, 42) $ & $ (420, 42) $ & $ (1680, 30) $ & $ (1680, 10) $\\
$1586  $ & $1 $ & $1 $ &  $2 $ & $ (660, 12) $ & $ (220, 2) $ & $ (132 , 6)
 $ & $ (132, 6) $ & $ (132, 6) $ & $ (88, 2) $ & $ (88, 2) $\\
$2290  $ & $1 $ & $-1 $ &  $2 $ & $ (780, 60) $ & $ (260, 2) $ & $ (60 , 10)
 $ & $ (60, 10) $ & $ (60, 10) $ & $ (120, 2) $ & $ (120, 2) $\\
$2626  $ & $2 $ & $1 $ &  $2 $ & $ (504, 12) $ & $ (252, 6) $ & $ (252 , 6)
 $ & $ (36, 6) $ & $ (36, 6) $ & $ (144, 6) $ & $ (144, 6) $\\
$4610 $ & $2 $ & $1 $ &  $4 $ & $ (18720, 12) $ & $ (2340, 30) $ & $ (2340 , 30) $ & $ (780, 30) $ & $ (780, 30) $ & $ (12480, 30) $ & $ (12480, 10) $\\
 $5410$ & $1 $ & $-1 $ &  $3 $ & $ (3960, 44) $ & $ (44, 22) $ & $ (44 , 22) $ & $ (220, 2) $ & $ (220, 2) $ & $ (1584, 2) $ & $ (1584, 2) $\\
 \hline
\end{longtable}
\begin{longtable}{| c   c  c  c c  c  c  c c c |}
\caption{\small{ Invariants of $\LL_j$ for  $(\frac{p_1}{p_2})=1$ and $(\frac{\pi_1}{\pi_3})=-1$}\label{33}}\\
\hline
 $d$   & $b$ & $m,$ $n$
             & $\mathbf{C}l(\LL_1)$ & $\mathbf{C}l(\LL_2)$ & $\mathbf{C}l(\LL_3)$ & $\mathbf{C}l(\LL_4)$
    			 & $\mathbf{C}l(\LL_5)$ & $\mathbf{C}l(\LL_6)$ & $\mathbf{C}l(\LL_7)$\\
\hline
\endfirsthead
\hline
 $d$  & $b$ & $m,$ $n$
             & $\mathbf{C}l(\LL_1)$ & $\mathbf{C}l(\LL_2)$ & $\mathbf{C}l(\LL_3)$ & $\mathbf{C}l(\LL_4)$
    			 & $\mathbf{C}l(\LL_5)$ & $\mathbf{C}l(\LL_6)$ & $\mathbf{C}l(\LL_7)$\\
\hline
\endhead
$1090  $ & $1 $ & $4$, $1 $ & $ (240, 6) $ & $ (30, 6, 2) $ & $ (6 , 6, 6)
 $ & $ (6, 6, 6) $ & $ (6, 6, 6) $ & $ (24, 12) $ & $ (48, 6) $\\
$1490 $ & $-1 $ & $3$, $1 $ & $ (504, 6) $ & $ (126, 6, 2) $ & $ (126 , 6, 2)
 $ & $ (18, 6, 6) $ & $ (18, 6, 6) $ & $ (72, 2) $ & $ (36, 12) $\\
$4082 $ & $-1 $ & $4$, $1 $ & $ (624, 78) $ & $ (26, 26, 2) $ & $ (78 , 6, 2)
 $ & $ (78, 6, 2) $ & $ (78, 6, 2) $ & $ (208, 2) $ & $ (104, 4) $\\
$4706 $ & $1 $ & $3$, $1 $ & $ (6120, 6) $ & $ (306, 6, 2) $ & $ (510 , 6, 2)
 $ & $ (510, 6, 2) $ & $ (510, 6, 2) $ & $ (204, 4) $ & $ (408, 2) $\\
$6994 $ & $-1 $ & $4$, $1 $ & $ (7920, 30) $ & $ (90, 10, 2) $ & $ (990 , 6, 2)
 $ & $ (990, 6, 2) $ & $ (990, 6, 2) $ & $ (720, 2) $ & $ (360, 4) $\\
$7474  $ & $-1 $ & $5$, $1 $ & $ (43680, 2) $ & $ (78, 6, 2) $ & $ (2730 , 2, 2)
 $ & $ (2730, 2, 2) $ & $ (2730, 2, 2) $ & $ (1248, 6) $ & $ (208, 4) $\\
\hline
\end{longtable}
\normalsize
\tiny
\begin{longtable}{| c | c | c | c | c |c | c | c | c| c | c |}
\caption{\small{ Invariants of $\KK_j$ for the case $(\frac{p_1}{p_2})=-1$}\label{34}}\\
\hline
 $d$  $ = 2.p_1.p_2$ & $q$ & $\pi$ & $m$
             & $\mathbf{C}l(\KK_1)$ & $\mathbf{C}l(\KK_2)$ & $\mathbf{C}l(\KK_3)$ & $\mathbf{C}l(\KK_4)$
			   & $\mathbf{C}l(\KK_5)$ & $\mathbf{C}l(\KK_6)$ & $\mathbf{C}l(\KK_7)$\\
\hline
\endfirsthead
\hline
 $d$  $ = 2.p_1.p_2$ & $q$ & $\pi$ & $m$
             & $\mathbf{C}l(\KK_1)$ & $\mathbf{C}l(\KK_2)$ & $\mathbf{C}l(\KK_3)$ & $\mathbf{C}l(\KK_4)$
  			  & $\mathbf{C}l(\KK_5)$ & $\mathbf{C}l(\KK_6)$ & $\mathbf{C}l(\KK_7)$\\
\hline
\endhead
$130 = 2.5.13 $ & $2 $ & $-1$ & $2 $ & $ (12, 2) $ & $ (4, 2) $ & $ (8, 4) $ & $ (4, 2) $ & $ (2, 2, 2) $ & $ (2, 2, 2) $ & $ (4, 2) $\\
$370 = 2.37.5 $ & $1 $ & $-1$ & $3 $ & $ (12, 2) $ & $ (60, 2) $ & $ (24, 4) $ & $ (12, 2) $ & $ (6, 2, 2) $ & $ (6, 2, 2) $ & $ (12, 2) $\\
$530 = 2.5.53 $ & $2 $ & $1$ & $2 $ & $ (84, 2) $ & $ (84, 2) $ & $ (56, 4) $ & $ (14, 2, 2) $ & $ (28, 2) $ & $ (28, 2) $ & $ (14, 2, 2)$\\
$1970 = 2.197.5 $ & $2 $ & $1$ & $2 $ & $ (260, 2) $ & $ (260, 2) $ & $ (312, 12) $ & $ (26, 2, 2) $ & $ (52, 2) $ & $ (52, 2) $ & $ (26, 2, 2)$\\
$2930 = 2.293.5 $ & $1 $ & $1$ & $3 $ & $ (252, 2) $ & $ (252, 2) $ & $ (56, 4) $ & $ (42, 2, 2) $ & $ (28, 2) $ & $ (28, 2) $ & $ (42, 2, 2) $\\
$3538 = 2.61.29 $ & $1 $ & $-1$ & $5 $ & $ (84, 2) $ & $ (420, 2) $ & $ (224, 4) $ & $ (28, 2) $ & $ (14, 2, 2) $ & $ (14, 2, 2) $ & $ (28, 2) $\\
$5570 = 2.5.557 $ & $1 $ & $-1$ & $4 $ & $ (660, 2) $ & $ (180, 6) $ & $ (240, 4) $ & $ (60, 2) $ & $ (30, 2, 2) $ & $ (30, 2, 2) $ & $ (60, 2) $\\
$6130 = 2.5.613 $ & $2 $ & $-1$ & $2 $ & $ (1260, 6) $ & $ (180, 2) $ & $ (72, 36) $ & $ (36, 2) $ & $ (18, 2, 2) $ & $ (18, 2, 2) $ & $ (36, 2) $\\
\hline
\end{longtable}
\begin{longtable}{| c  | c | c | c |c | c | c | c| c | c | c |}
\caption{\small{ Invariants of $\LL_j$ for the case $(\frac{p_1}{p_2})=-1$}\label{35}}\\
\hline
 $d$  $ = 2.p_1.p_2$ & $q$ & $\pi$ & $m$
             & $\mathbf{C}l(\LL_1)$ & $\mathbf{C}l(\LL_2)$ & $\mathbf{C}l(\LL_3)$ & $\mathbf{C}l(\LL_4)$
    			 & $\mathbf{C}l(\LL_5)$ & $\mathbf{C}l(\LL_6)$ & $\mathbf{C}l(\LL_7)$\\
\hline
\endfirsthead
\hline
 $d$  $ = 2.p_1.p_2$ & $q$ & $\pi$ & $m$
             & $\mathbf{C}l(\LL_1)$ & $\mathbf{C}l(\LL_2)$ & $\mathbf{C}l(\LL_3)$ & $\mathbf{C}l(\LL_4)$
    			 & $\mathbf{C}l(\LL_5)$ & $\mathbf{C}l(\LL_6)$ & $\mathbf{C}l(\LL_7)$\\
\hline
\endhead
$130 = 2.5.13 $ & $2 $ & $-1 $ & $2 $ & $ (24, 2) $ & $ (12, 2) $ & $ (12 , 2)
 $ & $ (4, 2) $ & $ (4, 2) $ & $ (8, 2) $ & $ (4, 4) $\\
$370 = 2.5.37 $ & $1 $ & $-1 $ & $3 $ & $ (120, 2) $ & $ (60, 2) $ & $ (12 , 2)
 $ & $ (12, 2) $ & $ (12, 2) $ & $ (24, 2) $ & $ (12, 4) $\\
$530 = 2.5.53 $ & $2 $ & $1 $ & $2 $ & $ (168, 6) $ & $ (84, 2) $ & $ (84 , 2)
 $ & $ (84, 2) $ & $ (84, 2) $ & $ (28, 4) $ & $ (56, 2) $\\
$1970 = 2.5.197 $ & $2 $ & $1 $ & $2 $ & $ (1560, 30) $ & $ (260, 2) $ & $ (260 , 2)
 $ & $ (260, 2) $ & $ (260, 2) $ & $ (156, 12) $ & $ (312, 6) $\\
$2930 = 2.293.5 $ & $1 $ & $1 $ & $3 $ & $ (504, 18) $ & $ (252, 6) $ & $ (252 , 6)
 $ & $ (252, 6) $ & $ (252, 6) $ & $ (84, 12) $ & $ (56, 2) $\\
$3538 = 2.29.61 $ & $1 $ & $-1 $ & $5 $ & $ (3360, 6) $ & $ (420, 2) $ & $ (420 , 2)
 $ & $ (84, 2) $ & $ (84, 2) $ & $ (224, 2) $ & $ (112, 4) $\\
$5570 = 2.557.5 $ & $1 $ & $-1 $ & $4 $ & $ (7920, 6) $ & $ (180, 6) $ & $ (660 , 2)
 $ & $ (660, 2) $ & $ (660, 2) $ & $ (240, 2) $ & $ (120, 4) $\\
$6130 = 2.5.613 $ & $2 $ & $-1 $ & $2 $ & $ (2520, 90) $ & $ (1260, 6) $ & $ (1260 , 6)
 $ & $ (180, 2) $ & $ (180, 2) $ & $ (72, 18) $ & $ (36, 36) $\\
\hline
\end{longtable}

\end{document}